\documentclass[11pt]{article}
\usepackage{amsfonts,amscd,amssymb, amsmath}
\usepackage[usenames]{color}
\usepackage{graphicx}
\usepackage{euscript}
\usepackage[all,tips]{xy}
\usepackage{epic,eepic}
\usepackage{color}
\usepackage{tikz}

\newtheorem{definition}{Definition}[section]
\newtheorem{lemma}[definition]{Lemma}
\newtheorem{proposition}[definition]{Proposition}
\newtheorem{corollary}[definition]{Corollary}
\newtheorem{remark}[definition]{Remark}

\newtheorem{theorem}[definition]{Theorem}

\def\rawo\lonra{\longrightarrow}

\def\ot{\otimes}

\allowdisplaybreaks[4]
\newenvironment{proof}{{\it Proof.}}{\hfill $ \square $ \vskip 4mm}

\begin{document}
\title{BiHom-Novikov algebras and infinitesimal BiHom-bialgebras}
\author{Ling Liu \\
College of Mathematics and Computer Science,\\
Zhejiang Normal University, \\
Jinhua 321004, China \\
e-mail: ntliulin@zjnu.cn \and Abdenacer Makhlouf \\
Universit\'{e} de Haute Alsace, \\
IRIMAS-D\'epartement  de Math\'{e}matiques,  \\
6 bis, rue des fr\`{e}res Lumi\`{e}re, F-68093 Mulhouse, France\\
e-mail: Abdenacer.Makhlouf@uha.fr \and Claudia Menini \\
University of Ferrara, \\
Department of Mathematics and Computer Science\\
Via Machiavelli 30, Ferrara, I-44121, Italy\\
e-mail: men@unife.it \and Florin Panaite \\
Institute of Mathematics of the Romanian Academy,\\
PO-Box 1-764, RO-014700 Bucharest, Romania\\
e-mail: florin.panaite@imar.ro }
\maketitle

\begin{abstract}
We introduce and study infinitesimal BiHom-bialgebras, BiHom-Novikov algebras, BiHom-Novikov-Poisson algebras, 
and find some relations among these concepts. Our main result is to show how to obtain a 
left BiHom-pre-Lie algebra from an infinitesimal BiHom-bialgebra. \\

\begin{small}
\end{small}
\end{abstract}

\section*{Introduction}

Algebras of Hom-type were introduced in the Physics literature of the 1990's related to quantum deformations 
of algebras of vector fields, which satisfy a modified Jacobi identity involving a homomorphism (such algebras were called 
Hom-Lie algebras in \cite{hls}, \cite{larsson}). Hom-analogues of other algebraic structures have been introduced afterwards, 
such as Hom-(co)associative (co)algebras, 
Hom-bialgebras, Hom-pre-Lie algebras etc. Recently, structures of a more general type have been 
introduced in \cite{gmmp}, called BiHom-type algebras, for which a certain algebraic identity is twisted by 
two commuting homomorphisms (called structure maps). 

Infinitesimal bialgebras were introduced by Joni and Rota in \cite{jonirota} and 
studied by Aguiar in a series of papers (\cite{aguiarcontemp,aguiarjalgebra, aguiarlectnotes}). They have connections 
with some other 
concepts such as Rota-Baxter operators, pre-Lie algebras, Lie bialgebras etc. A prominent class of examples, discovered by 
Aguiar, is provided by the path algebra of an arbitrary quiver. Infinitesimal bialgebras have been used (with a 
different name) in \cite{voiculescu} by Voiculescu 
in free probability theory. 

The concept of infinitesimal Hom-bialgebra (the Hom-analogue of infinitesimal bialgebras) 
was introduced and studied by Yau in \cite{yauinf}. We continued this study in our previous paper \cite{lmmp3}, 
where we obtained a Hom-analogue of the following theorem of Aguiar from \cite{aguiarlectnotes}:
Let $(A, \mu , \Delta )$ be an infinitesimal bialgebra, with notation $\mu (a\otimes b)=a\cdot b$ and 
$\Delta (a)=a_1\otimes a_2$, for all $a, b\in A$; if one defines a new operation on $A$ by $a\bullet b=b_1\cdot a\cdot b_2$, then 
$(A, \bullet $) is a left pre-Lie algebra. 
In order to obtain the Hom-generalization of this result, we relied on the observation that, if the infinitesimal bialgebra 
in Aguiar's theorem is commutative, then his theorem is a particular case of the Gel'fand-Dorfman theorem which shows how 
to obtain a Novikov algebra by using a derivation on a commutative associative algebra. Thus, what we did in \cite{lmmp3} 
was essentially to find a sort of connection between infinitesimal Hom-bialgebras and 
Hom-Novikov algebras (these have been also introduced by Yau in \cite{yaunovikov}). 

The main aim of the present paper is to extend the above mentioned results to the BiHom case. We define 
infinitesimal BiHom-bialgebras, and our main result (Theorem \ref{infprelie}) shows how to obtain a 
left BiHom-pre-Lie algebra from an infinitesimal BiHom-bialgebra. We use the same strategy as in the Hom case, namely we find 
a connection with so-called BiHom-Novikov algebras, that we also introduce and study here (we would like to 
emphasize that our concept of BiHom-Novikov algebra is different from the one introduced in \cite{guo}). 
Along the way, we also introduce and study several other concepts, such as BiHom-Novikov-Poisson algebras and 
quasitriangular infinitesimal BiHom-bialgebras. 

\section{Preliminaries} \label{sec1} 
\setcounter{equation}{0} 

We work over a base field $\Bbbk $. All
algebras, linear spaces etc. will be over $\Bbbk $; unadorned $\otimes $
means $\otimes_{\Bbbk}$. Unless otherwise specified, the
(co)algebras that will appear in what follows are
\emph{not} supposed to be (co)associative or 
(co)unital, the multiplication $\mu :A\otimes
A\rightarrow A$ of an algebra $(A, \mu )$ is denoted by $\mu
(a\otimes a^{\prime })=a\cdot a^{\prime }$, and for a comultiplication $\Delta :C\rightarrow
C\otimes C$ on a linear space $C$ we use a Sweedler-type notation $\Delta
(c)=c_1\otimes c_2$, for $c\in C$. 
For the composition of two maps $f$
and $g$, we will write either $g\circ f$ or simply $gf$. For the identity
map on a linear space $V$ we will use the notation $id_V$.

\begin{definition} (\cite{gmmp}) 
A BiHom-associative algebra is a 4-tuple $\left( A,\mu ,\alpha ,\beta \right) $, where $A$ is
a linear space and $\alpha , \beta :A\rightarrow A$ 
and $\mu :A\otimes A\rightarrow A$ are linear maps such that 
$\alpha \circ \beta =\beta \circ \alpha $, 
$\alpha (x\cdot y) =\alpha (x)\cdot \alpha (y)$, $\beta (x\cdot y)=\beta (x)\cdot \beta (y)$ and 
\begin{eqnarray}
\alpha (x)\cdot (y\cdot z)=(x\cdot y)\cdot \beta (z), \label{BHassoc}
\end{eqnarray}
for all $x, y, z\in A$. The maps $\alpha $ and $\beta $ (in this order) are called the structure maps
of $A$ and condition (\ref{BHassoc}) is called the BiHom-associativity condition.

A morphism $f:(A, \mu _A , \alpha _A, \beta _A)\rightarrow (B, \mu _B ,
\alpha _B, \beta _B)$ of BiHom-associative algebras is a linear map $%
f:A\rightarrow B$ such that $\alpha _B\circ f=f\circ \alpha _A$, $\beta
_B\circ f=f\circ \beta _A$ and $f\circ \mu_A=\mu _B\circ (f\otimes f)$.
\end{definition}

If $(A, \mu )$ is an associative algebra and $\alpha , \beta :A\rightarrow A$ are two commuting algebra maps, then 
$A_{(\alpha , \beta )}:=(A, \mu \circ (\alpha \otimes \beta ), \alpha , \beta )$ is a BiHom-associative algebra, 
called the Yau twist of $A$ via the maps $\alpha $ and $\beta $. 
\begin{definition} (\cite{gmmp}) 
A BiHom-coassociative coalgebra is a 4-tuple $(C, \Delta, \psi ,
\omega )$, in which $C$ is a linear space, $\psi , \omega :C\rightarrow C$
and $\Delta :C\rightarrow C\otimes C$ are linear maps, such that $\psi \circ \omega =\omega \circ \psi $, 
$(\psi \otimes \psi )\circ \Delta = \Delta \circ \psi $, 
$(\omega \otimes \omega )\circ \Delta = \Delta \circ \omega $ and 
\begin{eqnarray}
&&(\Delta \otimes \psi )\circ \Delta = (\omega \otimes \Delta )\circ \Delta . \label{BHcoassoc}
\end{eqnarray}

The maps $\psi $ and $\omega $ (in this order) are called the structure maps of
$C$ and condition (\ref{BHcoassoc}) is called the BiHom-coassociativity condition.

A morphism $g:(C, \Delta _C , \psi _C, \omega _C)\rightarrow (D, \Delta _D ,
\psi _D, \omega _D)$ of BiHom-coassociative coalgebras is a linear map $%
g:C\rightarrow D$ such that $\psi _D\circ g=g\circ \psi _C$, $\omega _D\circ
g=g\circ \omega _C$ and $(g\otimes g)\circ \Delta _C=\Delta _D\circ g$.
\end{definition}
\begin{definition}
A left pre-Lie algebra is a pair $(A, \mu )$, where $A$ is a a linear space and 
$\mu :A\ot A\rightarrow A$ is a linear map satisfying the condition 
\begin{eqnarray*}
&&x\cdot (y\cdot z)-(x\cdot y)\cdot z=y\cdot (x\cdot z)-(y\cdot x)\cdot z, \;\;\;  \forall \;\;x, y, z\in A. 
\end{eqnarray*}

A morphism of left pre-Lie algebras from $(A, \mu )$ to $(A', \mu ')$ is a linear map 
$f :A\rightarrow A'$ satisfying $f (x\cdot y)=f (x)\cdot ' f(y)$, for all $x, y\in A$. 
\end{definition}
\begin{definition} (\cite{lmmp2}) 
A left BiHom-pre-Lie algebra is a 4-tuple $(A, \mu ,
\alpha , \beta )$, where $A$ is a linear space and $\mu :A\otimes
A\rightarrow A$ and $\alpha , \beta :A\rightarrow A$ are linear maps
satisfying $\alpha \circ \beta =\beta \circ \alpha $, $\alpha (x\cdot
y)=\alpha (x)\cdot \alpha (y)$, $\beta (x\cdot y)=\beta (x)\cdot \beta (y)$
and
\begin{eqnarray}
&&\alpha \beta (x)\cdot (\alpha (y)\cdot z)-(\beta (x)\cdot \alpha (y))\cdot
\beta (z) =\alpha \beta (y)\cdot (\alpha (x)\cdot z)-(\beta (y)\cdot \alpha
(x))\cdot \beta (z),  \label{lBHpL}
\end{eqnarray}
for all $x, y, z\in A$. We call $\alpha $ and $\beta $ (in this order) the
structure maps of $A$.

A morphism $f:(A, \mu , \alpha , \beta )\rightarrow (A^{\prime }, \mu 
^{\prime }, \alpha ^{\prime }, \beta ^{\prime })$ of left 
BiHom-pre-Lie algebras is a linear map $f:A\rightarrow A^{\prime }$
satisfying $f(x\cdot y)=f(x)\cdot ^{\prime }f(y)$, for all $x, y\in A$, as
well as $f\circ \alpha =\alpha ^{\prime }\circ f$ and $f\circ \beta =\beta
^{\prime }\circ f$.
\end{definition}

A BiHom-associative algebra is an example of a left BiHom-pre-Lie algebra.

\begin{definition} (\cite{lmmp2}) 
A left (respectively right) BiHom-Leibniz algebra is a 4-tuple $(L, [\cdot
,\cdot ], \alpha , \beta )$, where $L$ is a linear space, $[\cdot , \cdot
]:L\times L\rightarrow L$ is a bilinear map and $\alpha , \beta
:L\rightarrow L$ are linear maps satisfying $\alpha \circ \beta =\beta \circ
\alpha $, $\alpha ([x, y])=[\alpha (x), \alpha (y)]$, $\beta ([x, y])=[\beta
(x), \beta (y)]$ and
\begin{eqnarray}
&&[\alpha \beta (x), [y, z]]=[[\beta (x), y], \beta (z)]+[\beta (y), [\alpha
(x), z]],  \label{leftBHleibniz}
\end{eqnarray}
respectively
\begin{eqnarray}
&&[[x, y], \alpha \beta (z)]=[[x, \beta (z)], \alpha (y)]+[\alpha (x), [y,
\alpha (z)]],  \label{rightBHleibniz}
\end{eqnarray}
for all $x, y, z\in L$. We call $\alpha $ and $\beta $ (in this order) the
structure maps of $L$.
\end{definition}

\begin{definition} (\cite{lmmp2}) 
A left (respectively right) BiHom-Lie algebra is a left (respectively right)
BiHom-Leibniz algebra $(L, [\cdot ,\cdot ], \alpha , \beta )$ satisfying the
BiHom-skew-symmetry condition
\begin{eqnarray}
&&[\beta (x), \alpha (y)]=-[\beta (y), \alpha (x)], \;\;\;\forall \;x, y\in
L.  \label{BHskewsym}
\end{eqnarray}
A morphism $f:( L,[\cdot ,\cdot ] ,\alpha ,\beta )\rightarrow ( L^{\prime },
[\cdot ,\cdot ]^{\prime },\alpha ^{\prime },\beta ^{\prime })$ of BiHom-Lie
algebras is a linear map $f:L\rightarrow L^{\prime }$ such that $\alpha
^{\prime }\circ f=f\circ \alpha $, $\beta ^{\prime }\circ f=f\circ \beta $
and $f([x, y])=[f(x), f(y)]^{\prime }$, for all $x, y\in L$.
\end{definition}
\begin{definition} (\cite{lmmp1}) A BiHom-dendriform algebra is a 5-tuple $(A, \prec , \succ , \alpha , \beta )$
consisting of a linear space $A$ and linear maps $\prec , \succ :A\otimes A\rightarrow A$ and
$\alpha , \beta :A\rightarrow A$ such that, for all $x, y, z\in A$: 
\begin{eqnarray}
&&\alpha \circ \beta =\beta \circ \alpha , \label{BiHomdend1} \\
&&\alpha (x\prec y)=\alpha (x)\prec \alpha (y), ~
\alpha (x\succ y)=\alpha (x)\succ \alpha (y), \label{BiHomdend3} \\
&&\beta (x\prec y)=\beta (x)\prec \beta (y), ~
\beta (x\succ y)=\beta (x)\succ \beta (y), \label{BiHomdend5} \\
&&(x\prec y)\prec \beta (z)=\alpha (x)\prec (y\prec z+y\succ z),  \label{BiHomdend6} \\
&&(x \succ y)\prec \beta (z)=\alpha (x)\succ (y\prec z), \label{BiHomdend7} \\
&&\alpha (x)\succ (y\succ z)=(x\prec y+x\succ y)\succ \beta (z). \label{BiHomdend8}
\end{eqnarray}

We call $\alpha $ and $\beta $ (in this order) the structure maps
of $A$.
\end{definition}
\begin{proposition} (\cite{lmmp2}) \label{BHdendpreLie}
Let $(A, \prec , \succ , \alpha , \beta )$ be a BiHom-dendriform algebra such that $\alpha $ and $\beta $
are bijective. Let $\star :A\ot A\rightarrow A$ be the linear map defined for all $x, y\in A$ by
\begin{eqnarray*}
&&x\star y=x\succ y-(\alpha ^{-1}\beta (y))\prec (\alpha \beta ^{-1}(x)). 
\end{eqnarray*}
Then $(A, \star, \alpha , \beta )$  
is a left BiHom-pre-Lie algebra.
\end{proposition}
\begin{definition} (\cite{gmmp}, \cite{lmmp1})
Let $(A, \mu _A , \alpha _A, \beta _A)$ be a BiHom-associative algebra and let $(M, \alpha _M, \beta _M)$ be a triple where $M$
is a linear space and $\alpha _M, \beta _M:M \rightarrow M$ are commuting linear maps.\\
(i) $(M, \alpha _M, \beta _M)$ is a left $A$-module if 
we have a linear map $A\otimes M\rightarrow M$, $a\otimes m\mapsto a\cdot m$, such that
$\alpha _M(a\cdot m)=\alpha _A(a)\cdot \alpha _M(m)$,
$\beta _M(a\cdot m)=\beta _A(a)\cdot \beta _M(m)$ and
\begin{eqnarray}
&&\alpha _A(a)\cdot (a^{\prime }\cdot m)=(a\cdot a^{\prime })\cdot \beta _M(m),\;\;\;\;\forall \;\;a, a'\in A, \;m\in M.
\label{lmod4}
\end{eqnarray}
(ii) $(M, \alpha _M, \beta _M)$ is a right $A$-module if 
we have a linear map $M\otimes A\rightarrow M$, $m\otimes a\mapsto m\cdot a$, such that
$\alpha _M(m\cdot a)=\alpha _M(m)\cdot \alpha _A(a)$,
$\beta _M(m\cdot a)=\beta _M(m)\cdot \beta _A(a)$ and
\begin{eqnarray}
&&\alpha _M(m)\cdot (a\cdot a')=(m\cdot a)\cdot \beta _A(a'), \;\;\;\;\forall \;\;a, a'\in A, \;m\in M.
\label{rmod4}
\end{eqnarray}
(iii) If $(M, \alpha _M, \beta _M)$ is 
a left and right $A$-module, 
then $M$ is called an $A$-bimodule if 
\begin{eqnarray}
&&\alpha _A(a)\cdot (m\cdot a')=(a\cdot m)\cdot \beta _A(a'), \;\;\;\;\forall \;\;a, a'\in A, \;m\in M. \label{BHbim}
\end{eqnarray}
\end{definition}
\begin{definition}
An algebra $(A, \mu )$ is called a Novikov algebra if it is left pre-Lie and 
\begin{eqnarray}
&&(x\cdot y)\cdot z=(x\cdot z)\cdot y, \;\;\; \forall \; x, y, z\in A. \label{Novikov}
\end{eqnarray}

A morphism of Novikov algebras from $(A, \mu )$ to $(A', \mu ')$ is a linear map 
$f :A\rightarrow A'$ satisfying $f (x\cdot y)=f (x)\cdot ' f(y)$, for all $x, y\in A$. 
\end{definition}
\begin{theorem} \label{GD}
(Gel'fand-Dorfman)
Let $(A, \mu )$ be an associative and commutative algebra and let $D:A\rightarrow A$ be 
a derivation. Define a new multiplication on $A$ by $a* b=a\cdot D(b)$, for all $a, b\in A$. 
Then $(A, * )$ is a Novikov algebra.
\end{theorem}
\begin{definition} (\cite{Xu1}, \cite{Xu2}) A Novikov-Poisson algebra is a triple $(A, \cdot , *)$ such that 
$(A, \cdot )$ is a  commutative associative algebra, $(A, * )$ is a Novikov algebra and 
the following compatibility conditions hold, for all $x, y, z\in A$:
 \begin{eqnarray}
&& (x*y)\cdot z-x* (y\cdot z)=(y*x)\cdot z-y* (x\cdot z),~~~~ \label{NP1}\\
 && (x \cdot y)* z=(x*z)\cdot y. \label{NP2}
\end{eqnarray}

A morphism of Novikov-Poisson algebras from $(A, \cdot , * )$ to $(A', \cdot ', *')$ is a linear map 
$f :A\rightarrow A'$ satisfying $f (x\cdot y)=f (x)\cdot ' f(y)$ and $f (x* y)=f (x)*' f(y)$, 
for all $x, y\in A$. 
\end{definition}

Note that, by the commutativity of $(A, \cdot )$, (\ref{NP2}) is equivalent to 
\begin{eqnarray}
&&(x\cdot y)*z=x\cdot (y*z), \;\;\; \forall \; x, y, z\in A. \label{NP3}
\end{eqnarray}

\section{BiHom-Novikov algebras} \label{sec2}
\setcounter{equation}{0} 
We begin by introducing the BiHom-analogue of Novikov algebras (the Hom-analogue was introduced in \cite{yaunovikov}). 
\begin{definition}
\label{binovialgebra} A BiHom-Novikov algebra is a 4-tuple $(A, \mu, \alpha
, \beta)$, where $A$ is a linear space, $\mu: A \otimes A\rightarrow A$ is a
linear map and $%
\alpha , \beta : A \rightarrow A$ are commuting linear maps (called the structure maps of $A$), 
satisfying the
following conditions, for all $x, y, z \in A$:
\begin{eqnarray}
&& \alpha(x\cdot y)=\alpha(x)\cdot \alpha(y),~~ \beta(x\cdot y)=\beta(x)\cdot \beta(y),
\label{BiNovik} \\
&&
(\beta(x)\cdot \alpha(y))\cdot \beta(z)-\alpha\beta(x)\cdot (\alpha(y)\cdot z)=(\beta(y)\cdot \alpha(x))\cdot 
\beta(z)-\alpha\beta(y)\cdot (\alpha(x)\cdot z),  \label{BiNoviko} \\
&& (x \cdot \beta (y))\cdot \alpha\beta(z)=(x \cdot \beta (z))\cdot \alpha\beta(y).
\label{Binovikov}
\end{eqnarray}
In other words, a BiHom-Novikov algebra is a left BiHom-pre-Lie algebra
satisfying (\ref{Binovikov}).

A morphism $f:(A, \mu _A , \alpha _A, \beta _A)\rightarrow (B, \mu _B ,
\alpha _B, \beta _B)$ of BiHom-Novikov algebras is a linear map 
$f:A\rightarrow B$ such that $\alpha _B\circ f=f\circ \alpha _A$, $\beta
_B\circ f=f\circ \beta _A$ and $f\circ \mu_A=\mu _B\circ (f\otimes f)$.
\end{definition}

\begin{proposition}Let $(A, \mu )$ be a Novikov algebra and let $\alpha , \beta :
A\rightarrow A$ be two commuting Novikov algebra morphisms. Then $A_{(\alpha ,
\beta )}:=(A, \mu_{(\alpha , \beta )}:=\mu \circ (\alpha \otimes \beta),
\alpha , \beta )$ is a BiHom-Novikov algebra, called the Yau twist of $(A,
\mu )$.
\end{proposition}

\begin{proof}
For all $x, y\in A$, we write $\mu_{(\alpha , \beta )}(x \otimes y)=x \star y
= \alpha(x)\cdot \beta(y)$. We already know from \cite%
{lmmp2} that $A_{(\alpha , \beta )}$ is a left BiHom-pre-Lie algebra, so we
only need to prove (\ref{Binovikov}). We compute:
\begin{eqnarray*}
(x \star \beta (y))\star \alpha\beta(z)&=& (\alpha^2
(x)\cdot \alpha\beta^{2}(y))\cdot \alpha\beta^{2} (z) \\
&\overset{(\ref{Novikov})}{=}&(\alpha^2 (x) \cdot \alpha
\beta^{2}(z))\cdot \alpha\beta^{2}(y) \\
&=&\alpha(\alpha (x)\cdot \beta^{2}(z))\cdot \beta (\alpha\beta(y)) \\
&=& (\alpha(x)\cdot \beta^2(z))\star \alpha\beta(y)=(x \star \beta (z))\star
\alpha\beta(y).
\end{eqnarray*}
So indeed $A_{(\alpha , \beta )}$ is a BiHom-Novikov algebra.
\end{proof}

More generally, one can prove the following result: 
\begin{proposition}
Let $(A, \mu, \alpha, \beta)$ be a BiHom-Novikov algebra and let $\tilde{\alpha}, \tilde{\beta}: A\rightarrow A$ be 
two morphisms of BiHom-Novikov algebras such that any two of the maps $\alpha $, $\beta $, $\tilde{\alpha}$, 
$\tilde{\beta}$ commute. Then
$A_{(\tilde{\alpha}, \tilde{\beta})}:=(A, \mu \circ (\tilde{\alpha}\otimes \tilde{\beta}), 
\alpha\circ \tilde{\alpha},  \beta\circ \tilde{\beta})$ is also a BiHom-Novikov algebra.
\end{proposition}

The next concept is the BiHom-analogue of the concept of commutative associative algebra. 

\begin{definition}
\label{commbinovi} A BiHom-associative algebra $(A, \mu, \alpha , \beta)$ is
called BiHom-commutative if
\begin{eqnarray}
\beta(a)\cdot \alpha(b)=\beta(b)\cdot \alpha(a), \;\;\; \forall \;\;a, b\in A. 
\label{commnovi}
\end{eqnarray}
\end{definition}
\begin{remark}
If $(A, \mu)$ is a commutative associative algebra and $\alpha , \beta
:A\rightarrow A$ are commuting algebra morphisms, then $(A, \mu \circ (\alpha \otimes \beta ), \alpha ,
\beta )$ is a BiHom-commutative algebra. More generally, if $(A, \mu, \alpha, \beta)$ is a BiHom-commutative algebra and $\tilde{\alpha}, \tilde{\beta}: A\rightarrow A$ 
are morphisms of BiHom-associative algebras such that any two of the maps $\alpha $, $\beta $, $\tilde{\alpha}$, 
$\tilde{\beta}$ commute, then $(A, \mu \circ (\tilde{\alpha}\otimes \tilde{\beta}), 
\alpha\circ \tilde{\alpha},  \beta\circ \tilde{\beta})$ is a BiHom-commutative algebra.
\end{remark}
\begin{remark} \label{BHCBHN}
Obviously, any BiHom-commutative algebra is BiHom-Novikov.
\end{remark}

Our next result is the BiHom-analogue of the Gel'fand-Dorfman Theorem \ref{GD} (for a Hom-analogue see \cite{yaunovikov}).  
\begin{proposition}\label{generalGD}
Let $(A,\mu ,\alpha ,\beta )$ be a BiHom-commutative algebra. 
Let $\gamma , \lambda , \xi :A\rightarrow A$ be linear maps such that 
$\gamma (x\cdot y)=\gamma (x)\cdot \gamma (y)$, $\lambda (x\cdot y)=\lambda (x)\cdot \lambda (y)$ and 
$\xi (x\cdot y)=\xi (x)\cdot \xi (y)$, for all $x, y\in A$. 
Let $D:A\rightarrow A$ be a linear map, assume that any two of the maps $\alpha , \beta , \gamma , \lambda , 
\xi , D$ commute and the following condition is satisfied: 
\begin{eqnarray}
&&D\left( a\cdot b\right) =\gamma \left( a\right) \cdot D(b)+D\left( a\right) \cdot
\gamma \left( b\right), \;\;\;\forall \;\;a, b\in A. \label{gammaderivation}
\end{eqnarray}
Define a new multiplication on $A$ by 
$a\ast b=\lambda \left( a\right) \cdot \xi D\left( b\right)$, for all $a, b\in A$. 
Then $\left( A,\ast ,\lambda \alpha ,\xi \beta \gamma \right) $ is a
BiHom-Novikov algebra.
\end{proposition}
\begin{proof} It is easy to see that $\lambda \alpha (a\ast b)=\lambda \alpha (a)\ast \lambda \alpha (b)$ and 
$\xi \beta \gamma (a\ast b)=\xi \beta \gamma (a)\ast \xi \beta \gamma (b)$, for all $a, b\in A$. Now we compute: \\[2mm]
${\;\;\;\;\;}$
$\lambda \alpha \xi \beta \gamma (x)\ast (\lambda \alpha (y)\ast z)$
\begin{eqnarray*}
&=&\lambda \alpha \xi \beta \gamma (x)\ast \left( \lambda \lambda \alpha
(y)\cdot \xi D\left( z\right) \right)  
=\lambda ^{2}\alpha \xi \beta \gamma (x)\cdot \xi D\left( \lambda
^{2}\alpha (y)\cdot \xi D\left( z\right) \right)  \\
&\overset{(\ref{gammaderivation})}{=}&\lambda ^{2}\alpha \xi \beta \gamma (x)\cdot \xi \left[ \gamma \left(
\lambda ^{2}\alpha (y)\right) \cdot D(\xi D\left( z\right) )+D\left( \lambda
^{2}\alpha (y)\right) \cdot \gamma \left( \xi D\left( z\right) \right) %
\right]  \\
&=&\lambda ^{2}\alpha \xi \beta \gamma (x)\cdot \left[ \xi \gamma \lambda
^{2}\alpha (y)\cdot \xi ^{2}D^{2}\left( z\right) +\xi \lambda ^{2}\alpha
D(y)\cdot \gamma \xi ^{2}D\left( z\right) \right]  \\
&=&\lambda
^{2}\alpha \xi \beta \gamma (x)\cdot \left( \xi \gamma \lambda ^{2}\alpha
(y)\cdot \xi ^{2}D^{2}\left( z\right) \right)+
\lambda ^{2}\alpha \xi \beta \gamma (x)\cdot \left( \xi \lambda
^{2}\alpha D(y)\cdot \gamma \xi ^{2}D\left( z\right) \right), 
\end{eqnarray*}%
\begin{eqnarray*}
\left( \xi \beta \gamma (x)\ast \lambda \alpha (y)\right) \ast \xi \beta
\gamma \left( z\right)  &=&\left( \lambda \xi \beta \gamma (x)\cdot \lambda
\alpha \xi D(y)\right) \ast \xi \beta \gamma \left( z\right)  \\
&=&\left( \lambda ^{2}\xi \beta \gamma (x)\cdot \lambda ^{2}\alpha \xi
D(y)\right) \cdot \xi ^{2}\beta \gamma D\left( z\right),
\end{eqnarray*}
so that we get\\[2mm]
${\;\;\;\;\;}$
$\left( \xi \beta \gamma (x)\ast \lambda \alpha (y)\right) \ast \xi \beta
\gamma \left( z\right) -\lambda \alpha \xi \beta \gamma (x)\ast (\lambda
\alpha (y)\ast z)$
\begin{eqnarray*}
&=&\left( \lambda ^{2}\xi \beta \gamma (x)\cdot \lambda
^{2}\alpha \xi D(y)\right) \cdot \xi ^{2}\beta \gamma D\left( z\right)  
-\lambda ^{2}\alpha \xi
\beta \gamma (x)\cdot \left( \xi \gamma \lambda ^{2}\alpha (y)\cdot \xi
^{2}D^{2}\left( z\right) \right)  \\
&&-\lambda ^{2}\alpha \xi \beta \gamma (x)\cdot \left( \xi \lambda ^{2}\alpha
D(y)\cdot \gamma \xi ^{2}D\left( z\right) \right) \\
&=&\left( \lambda ^{2}\xi \beta \gamma (x)\cdot \lambda ^{2}\alpha \xi
D(y)\right) \cdot \beta \gamma \xi ^{2}D\left( z\right) 
-\lambda ^{2}\alpha \xi \beta \gamma
(x)\cdot \left( \xi \gamma \lambda ^{2}\alpha (y)\cdot \xi ^{2}D^{2}\left(
z\right) \right)  \\
&&-\left( \lambda
^{2}\xi \beta \gamma (x)\cdot \lambda ^{2}\alpha \xi D(y)\right) \cdot \beta
\gamma \xi ^{2}D\left( z\right)\\
&=&-\lambda ^{2}\alpha \xi \beta \gamma (x)\cdot \left( \xi \gamma \lambda
^{2}\alpha (y)\cdot \xi ^{2}D^{2}\left( z\right) \right)=
-\left( \lambda
^{2}\xi \beta \gamma (x)\cdot \xi \gamma \lambda ^{2}\alpha (y)\right) \cdot
\beta \xi ^{2}D^{2}\left( z\right)\\
&=&-\left( \beta (\lambda
^{2}\xi \gamma (x))\cdot \alpha (\xi \gamma \lambda ^{2}(y))\right) \cdot
\beta \xi ^{2}D^{2}\left( z\right)\\
&\overset{(\ref{commnovi})}{=}&-\left( \beta \lambda
^{2}\xi \gamma (y)\cdot \alpha \xi \gamma \lambda ^{2}(x)\right) \cdot
\beta \xi ^{2}D^{2}\left( z\right), 
\end{eqnarray*}
so the expression is symmetric in $x$ and $y$, meaning that $\left( A,\ast ,\lambda \alpha ,\xi \beta \gamma \right) $ is 
a left BiHom-pre-Lie algebra. We have to check now the BiHom-Novikov condition, namely 
$(x\ast \xi \beta \gamma (y))\ast \lambda \alpha
\xi \beta \gamma (z)=(x\ast \xi \beta \gamma (z))\ast \lambda \alpha \xi
\beta \gamma (y)$.  Note first that, since $(A,\mu ,\alpha ,\beta )$ is BiHom-commutative, it is BiHom-Novikov, so 
(\ref{Binovikov}) holds. 
Now we compute:
\begin{eqnarray*}
(x\ast \xi \beta \gamma (y))\ast \lambda \alpha \xi \beta \gamma (z)
&=&\left( \lambda \left( x\right) \cdot \xi D\left( \xi \beta \gamma
(y)\right) \right) \ast \lambda \alpha \xi \beta \gamma (z) \\
&=&\left[ \lambda ^{2}\left( x\right) \cdot \lambda \left( D\xi ^{2}\beta
\gamma (y)\right) \right] \cdot D\left( \alpha \beta \xi ^{2}\gamma \lambda
(z)\right)  \\
&=&\left[ \lambda ^{2}\left( x\right) \cdot \beta \left( \lambda D\xi
^{2}\gamma (y)\right) \right] \cdot \alpha \beta \left( D\xi
^{2}\gamma \lambda (z)\right)  \\
&\overset{(\ref{Binovikov})}{=}&\left[ \lambda ^{2}\left( x\right) \cdot \beta \left( \lambda D\xi
^{2}\gamma (z)\right) \right] \cdot \alpha \beta \left( D\xi
^{2}\gamma \lambda (y)\right) \\
&=&(x\ast \xi \beta \gamma (z))\ast \lambda \alpha \xi \beta \gamma (y),
\end{eqnarray*}%
finishing the proof.
\end{proof}
\begin{corollary}
Let $(A,\mu ,\alpha ,\beta )$ be a BiHom-commutative algebra. 
Let $p$ and $r$ be some natural numbers and let $D:A\rightarrow A$ be a linear map commuting with $\alpha $ and 
$\beta $ and satisfying the condition 
\begin{eqnarray*}
D\left( ab\right) =\beta ^{r}\left( a\right) \cdot D(b)+D\left( a\right)
\cdot \beta ^{r}\left( b\right), \;\;\; \forall \;\;a, b\in A.
\end{eqnarray*}
Define a new multiplication on $A$ by 
$a\ast b=\alpha ^{p}\left( a\right) \cdot D\left( b\right)$, for all $a, b\in A$. 
Then $\left( A,\ast ,\alpha ^{p+1},\beta ^{r+1}\right) $ is a BiHom-Novikov
algebra.
\end{corollary}
\begin{proof}
Take in the previous proposition $\lambda =\alpha ^p$, $\gamma =\beta ^r$, $\xi =id_A$. 
\end{proof}

In particular, by taking $p=r=0$, we obtain:
\begin{corollary} \label{commhom} 
Let $(A, \mu, \alpha , \beta)$ be a BiHom-commutative
algebra 
and let $D: A \rightarrow A$ be a derivation in the usual sense (i.e. $D(x\cdot y)=x\cdot D(y)+D(x)\cdot y$, for all $x, y\in A$) 
commuting with $\alpha $ and $\beta $. Then $(A, \ast , \alpha ,
\beta)$ is a BiHom-Novikov algebra, where
$a\ast b=a\cdot D(b)$, for all $a, b\in A$.
\end{corollary}

The following result is a consequence of Corollary \ref{commhom}.

\begin{corollary}
Let $(A, \mu )$ be an associative and commutative algebra, let $\alpha , \beta
:A\rightarrow A$ be two commuting algebra morphisms and let $D:A\rightarrow A$ be a
derivation such that $D\circ \alpha=\alpha \circ D$ and $D \circ \beta=
\beta \circ D$. Then $(A, \ast , \alpha , \beta)$ is a BiHom-Novikov
algebra, where $a\ast b=\alpha (a)\cdot D(\beta(b))$, for all $a, b\in A$.
\end{corollary}

\begin{proof}
The only thing that needs to be proved is that $D$ is also a derivation for $%
(A, \star )$, where $a\star b=\alpha (a)\cdot \beta (b)$, and this follows by a
straightforward computation.
\end{proof}

By taking in Proposition \ref{generalGD} $\alpha =\beta =\lambda =\xi =id_A$, we obtain the following result: 
\begin{proposition}
Let $(A,\mu )$ be a commutative associative algebra, let $\gamma :A\rightarrow A$ be an algebra map and let 
$D:A\rightarrow A$ be a $(\gamma , \gamma )$-derivation, i.e. 
\begin{eqnarray}
&&D\left( a\cdot b\right) =\gamma \left( a\right) \cdot D(b)+D\left( a\right) \cdot
\gamma \left( b\right), \;\;\;\forall \;\;a, b\in A. 
\end{eqnarray}
Assume that moreover we have $D\circ \gamma =\gamma \circ D$. 
Define a new multiplication on $A$ by 
$a* b=a \cdot D(b)$, for all $a, b\in A$. 
Then $( A,* , id_A, \gamma )$ is a
BiHom-Novikov algebra.
\end{proposition}

For a left BiHom-Lie algebra $(L, [\cdot ,\cdot ], \alpha , \beta)$, we
write $Z_{l}(\beta(L))$ for the subset of $L$ consisting of the elements $x
\in L$ such that $[x, \beta(y)]=0$ for all $y \in L$. Similarly, we write $%
Z_{r}(\alpha(L))$ for the subset of $L$ consisting of the elements $x \in L$
such that $[\alpha(y), x]=0$ for all $y \in L$.

\begin{proposition}
\label{bilienovikov} Let $(L, [\cdot ,\cdot ], \alpha , \beta)$ be a left
BiHom-Lie algebra and let $f: L \rightarrow L$ be a linear map such that $f\circ
\alpha=\alpha \circ f$ and $f \circ \beta= \beta \circ f$. Define two
operations on $A$ by
\begin{eqnarray*}
&& x\star y=[f(x),y] \;\;\;and \;\;\;x\star^{\prime }y=[x, f(y)],
\;\;\;\forall \;\;x, y\in A.
\end{eqnarray*}
Then we have:

(i) $(L, \star, \alpha , \beta)$ is a BiHom-Novikov algebra if and only if
the following conditions hold for all $x, y, z \in L$:
\begin{eqnarray}
&& f([f(\beta(x)), \alpha(y)]+[\beta (x), f(\alpha (y))])-[f(\beta(x)),
f(\alpha(y))]\in Z_{l}(\beta(L)),  \label{LieNo} \\
&& [f([f(x), \beta (y)]), \alpha\beta(z)]=[f([f(x), \beta (z)]), \alpha
\beta(y)] .  \label{LieNovi}
\end{eqnarray}

(ii) If $\alpha $ and $\beta $ are bijective, then $(L, \star^{\prime },
\alpha , \beta)$ is a BiHom-Novikov algebra if and only if the following
conditions hold for all $x, y, z \in L$:
\begin{eqnarray}
&& [[\beta (x), f(\alpha(y))]+[f(\beta (x)), \alpha (y)],
f(\beta(z))]-[\alpha \beta (x), f([\alpha (y), f(z)])]  \notag \\
&& \;\;\;\;\;\;\;\;+[\alpha \beta (y), f([\alpha (x), f(z)])]=0,
\label{LieNoviko} \\
&& [f(\beta(x)), f(\alpha(y))]\in Z_{r}(\alpha(L)).  \label{LieNovikov}
\end{eqnarray}
\end{proposition}

\begin{proof}
First note that $\alpha , \beta$ are multiplicative with respect to both $%
\star$ and $\star^{\prime }$ since $\alpha \circ [\cdot ,\cdot ]=[\cdot
,\cdot ]\circ (\alpha \otimes \alpha)$, $\beta \circ [\cdot ,\cdot ]=[\cdot
,\cdot ]\circ (\beta \otimes \beta)$ and $f\circ \alpha=\alpha \circ f$, $f
\circ \beta= \beta \circ f$.

Consider the first assertion. The condition (\ref{BiNoviko}) for the
multiplication $\star$ is
\begin{eqnarray*}
&& (\beta(x) \star \alpha(y))\star \beta(z)-\alpha\beta(x)\star
(\alpha(y)\star z)=(\beta(y)\star \alpha(x))\star
\beta(z)-\alpha\beta(y)\star (\alpha(x)\star z),
\end{eqnarray*}
and this is equivalent to
\begin{eqnarray*}
&& [f([f(\beta(x)), \alpha(y)]), \beta(z)]-[f(\alpha \beta (x)), [f(\alpha
(y)), z]] \\
&&\;\;\;\;\;\;\;\;= [f([f(\beta(y)), \alpha(x)]), \beta(z)]-[f(\alpha \beta
(y)), [f(\alpha (x)), z]],
\end{eqnarray*}
which, by using BiHom-skew-symmetry, is equivalent to
\begin{eqnarray*}
&& [f([f(\beta(x)), \alpha(y)]+[\beta(x), f(\alpha (y))]), \beta
(z)]-[f(\alpha \beta (x)), [f(\alpha (y)), z]] \\
&& \;\;\;\;\;\;+[f(\alpha \beta (y)), [f(\alpha (x)), z]]=0,
\end{eqnarray*}
which, by using (\ref{leftBHleibniz}), is equivalent to
\begin{eqnarray*}
&& [f([f(\beta(x)), \alpha(y)]+[\beta (x), f(\alpha (y))])-[f(\beta (x)),
f(\alpha (y))], \beta (z)]=0,
\end{eqnarray*}
and this is indeed equivalent to
\begin{eqnarray*}
f([f(\beta(x)), \alpha(y)]+[\beta (x), f(\alpha (y))])-[f(\beta(x)),
f(\alpha(y))]\in Z_{l}(\beta(L)), \;\;\; \forall \;\;x, y\in L.
\end{eqnarray*}
The condition (\ref{Binovikov}) for the multiplication $\star$ is $(x \star
\beta (y))\star \alpha\beta(z)=(x \star \beta(z))\star \alpha\beta(y)$, and
this is obviously equivalent to
\begin{eqnarray*}
[f([f(x), \beta (y)]), \alpha\beta(z)]=[f([f(x), \beta (z)]), \alpha
\beta(y)].
\end{eqnarray*}
This proves (i). To prove (ii), we need a result from \cite{lmmp2} saying
that, since $\alpha $ and $\beta $ are bijective, $(L, [\cdot ,\cdot ],
\alpha , \beta)$ is also a right BiHom-Lie algebra. One can easily see that (%
\ref{LieNoviko}) is just a restatement of (\ref{BiNoviko}) for the
multiplication $\star^{\prime }$. On the other hand, for all $x, y, z\in L$
we have:
\begin{eqnarray*}
(x \star^{\prime }\beta (y))\star^{\prime }\alpha\beta(z) &=& [[x,
f(\beta(y))], f(\alpha \beta (z))] \\
&\overset{(\ref{rightBHleibniz})}{=}& [[x, f(\beta(z))], f(\alpha \beta
(y))]+[\alpha(x), [f(\beta(y)), f(\alpha(z))]] \\
&=& (x \star^{\prime }\beta(z))\star^{\prime }\alpha\beta(y)+[\alpha(x),
[f(\beta(y)), f(\alpha(z))]].
\end{eqnarray*}
This implies that $(x \star^{\prime }\beta (y))\star^{\prime
}\alpha\beta(z)=(x \star^{\prime }\beta(z))\star^{\prime }\alpha\beta(y)$
holds for all $x, y, z\in L$ if and only if $[\alpha (x), [f(\beta (y)),
f(\alpha (z))]]=0$, or equivalently $[f(\beta (y)), f(\alpha (z))]\in
Z_{r}(\alpha(L))$ for all $y, z\in L$.
\end{proof}
\section{BiHom-Novikov-Poisson algebras}\label{sec3}
\setcounter{equation}{0}

The main aim of this section is to show that the BiHom-Novikov algebras obtained in Corollary  \ref{commhom} 
satisfy more compatibility conditions, turning them into what we will call BiHom-Novikov-Poisson algebras, which are the 
BiHom-analogues of Novikov-Poisson algebras.  Note that the Hom-analogue of Novikov-Poisson algebras has been introduced 
by Yau in \cite{yau3}. 
\begin{definition}\label{binovipoisson algebra}
A BiHom-Novikov-Poisson algebra is a 5-tuple $(A, \cdot, \ast, \alpha, \beta)$ such that:

(1) $(A, \cdot,  \alpha , \beta)$ is a  BiHom-commutative algebra;

(2) $(A, \ast,  \alpha , \beta)$ is a BiHom-Novikov algebra; 

(3) the following compatibility conditions hold for all $x, y, z\in A$:
 \begin{eqnarray}
&& (\beta(x)\ast \alpha(y))\cdot \beta(z)-\alpha\beta(x)\ast (\alpha(y)\cdot z)=(\beta(y)\ast \alpha(x))\cdot \beta(z)-\alpha\beta(y)\ast (\alpha(x)\cdot z),~~~~ \label{4.1}\\
 && (x \cdot \beta (y))\ast \alpha\beta(z)=(x\ast \beta (z))\cdot \alpha\beta(y),   \label{4.2} \\
&& \alpha (x)\cdot (y\ast z)=(x\cdot y)\ast \beta (z). \label{new}
\end{eqnarray}

The maps $\alpha $ and $\beta $ (in this order) are called the structure maps of $A$. 

A morphism $f:(A, \cdot, \ast, \alpha, \beta)\rightarrow (A', \cdot ', \ast ', \alpha ', \beta ')$ of BiHom-Novikov-Poisson 
algebras is a map that is a morphism of BiHom-associative algebras from $(A, \cdot , \alpha , \beta )$ to 
$(A', \cdot ', \alpha ', \beta ')$ and a morphism of BiHom-Novikov algebras from $(A, \ast , \alpha , \beta )$ to 
$(A', \ast ', \alpha ', \beta ')$. 
\end{definition}

Our first result shows that, in case of bijective structure maps, (\ref{4.2}) and (\ref{new}) are equivalent. 
\begin{lemma}\label{lemma 3.1}
Let $A$ be a linear space endowed with two linear multiplications $\cdot, \ast: A \otimes A\rightarrow A$ and two 
commuting bijective linear maps $\alpha , \beta : A \rightarrow A$ that are multiplicative with respect to $\cdot $ and $\ast $ 
and such that $(A, \cdot , \alpha , \beta )$ is a BiHom-commutative algebra. Then
\begin{eqnarray}
&& (x \cdot \beta (y))\ast \alpha\beta(z)=(x\ast \beta (z))\cdot \alpha\beta(y)   \label{4.3}
\end{eqnarray}
holds for all $x, y, z\in A$ if and only if
\begin{eqnarray}
 \alpha (x)\cdot (y\ast z)=(x\cdot y)\ast \beta (z)  \label{4.4}
\end{eqnarray}
holds for all $x, y, z\in A$.
\end{lemma}
\begin{proof}
We prove that (\ref{4.3}) implies (\ref{4.4}):  
\begin{eqnarray*}
 (x \cdot y)\ast \beta(z)
&=& (\beta (\beta^{-1}(x)) \cdot \alpha(\alpha^{-1}(y)))\ast \beta(z)\\
&\overset{(\ref{commnovi})}{=}&(\beta (\alpha^{-1}(y)) \cdot \alpha(\beta^{-1}(x)))\ast \beta(z)\\
&=& (\alpha^{-1}\beta (y)\cdot \beta(\alpha \beta^{-2}(x)))\ast \alpha \beta(\alpha ^{-1}(z))\\
&\overset{(\ref{4.3})}{=}& (\alpha^{-1}\beta (y) \ast \beta(\alpha ^{-1}(z)))\cdot \alpha \beta (\alpha \beta^{-2}(x)) \\
&=& \beta(\alpha^{-1}(y)\ast \alpha ^{-1}(z))\cdot \alpha(\alpha \beta^{-1}(x)) \\
&\overset{(\ref{commnovi})}{=}& \beta(\alpha \beta^{-1}(x))\cdot \alpha(\alpha^{-1}(y)\ast \alpha ^{-1}(z))\\
&=& \alpha(x) \cdot (y\ast z), \;\;\;q.e.d.
\end{eqnarray*}
The proof of the fact that (\ref{4.4}) implies (\ref{4.3}) is similar and left to the reader. 
\end{proof}
\begin{proposition}\label{Proposition 1}
 Let $(A, \cdot, \alpha, \beta)$ be a BiHom-commutative algebra. 
Then $(A, \cdot, \cdot, \alpha, \beta)$ is a BiHom-Novikov-Poisson algebra.
\end{proposition}
\begin{proof}
By Remark \ref{BHCBHN} we know that $(A, \cdot, \alpha, \beta)$ is a BiHom-Novikov algebra. 
The condition (\ref{4.1}) coincides with the condition (\ref{BiNoviko}), the condition (\ref{4.2}) coincides with the condition 
(\ref{Binovikov}), while (\ref{new}) is just the BiHom-associativity of $\cdot $. 
\end{proof}

The Hom-version of our next result may be found in \cite{yuanyou}. 
\begin{proposition} \label{BHGelD}
Under the hypotheses of Corollary  \ref{commhom} , $(A, \mu , *, \alpha , \beta )$ is a BiHom-Novikov-Poisson algebra. 
\end{proposition}
\begin{proof}
We only need to prove the relations (\ref{4.1}), (\ref{4.2}) and (\ref{new}). To prove (\ref{4.1}) we compute: \\[2mm]
${\;\;\;\;\;\;\;}$$(\beta (x)*\alpha (y))\cdot \beta (z)-\alpha \beta (x)*(\alpha (y)\cdot z)$
\begin{eqnarray*}
&=&(\beta (x)\cdot \alpha D(y))\cdot \beta (z)-\alpha \beta (x)\cdot D(\alpha (y)\cdot z)\\
&=&(\beta (x)\cdot \alpha D(y))\cdot \beta (z)-\alpha \beta (x)\cdot (\alpha (y)\cdot D(z)+\alpha D(y)\cdot z)\\
&=&(\beta (x)\cdot \alpha D(y))\cdot \beta (z)-\alpha \beta (x)\cdot (\alpha (y)\cdot D(z))
-\alpha \beta (x)\cdot (\alpha D(y)\cdot z)\\
&\overset{(\ref{BHassoc})}{=}& (\beta (x)\cdot \alpha D(y))\cdot \beta (z)-\alpha \beta (x)\cdot (\alpha (y)\cdot D(z))
-(\beta (x)\cdot \alpha D(y))\cdot \beta (z)\\
&=&-\alpha \beta (x)\cdot (\alpha (y)\cdot D(z))\\
&\overset{(\ref{BHassoc})}{=}&-(\beta (x)\cdot \alpha (y))\cdot \beta D(z),
\end{eqnarray*}
and this expression is obviously symmetric in $x$ and $y$ because of the BiHom-commutativity of $(A, \cdot , \alpha , \beta )$. 
To prove (\ref{4.2}) we compute: 
\begin{eqnarray*}
(x\cdot \beta (y))*\alpha \beta (z)&=&(x\cdot \beta (y))\cdot D\alpha \beta (z)\\
&\overset{(\ref{BHassoc})}{=}&\alpha (x)\cdot (\beta (y)\cdot \alpha (D(z)))\\
&\overset{(\ref{commnovi})}{=}&\alpha (x)\cdot (\beta D(z)\cdot \alpha (y))\\
&\overset{(\ref{BHassoc})}{=}&(x\cdot \beta D(z))\cdot \alpha \beta (y)
=(x\ast \beta (z))\cdot \alpha \beta (y). 
\end{eqnarray*}
Finally, to prove (\ref{new}) we compute: 
\begin{eqnarray*}
(x\cdot y)\ast \beta (z)&=&(x\cdot y)\cdot D\beta (z)\\
&\overset{(\ref{BHassoc})}{=}&\alpha (x)\cdot (y\cdot D(z))=\alpha (x)\cdot (y\ast z), 
\end{eqnarray*}
finishing the proof. 
\end{proof}

\begin{proposition} \label{YautwistNP}
Let $(A, \cdot , * )$ be a Novikov-Poisson algebra and let $\alpha , \beta :
A\rightarrow A$ be two commuting morphisms of Novikov-Poisson algebras. Then 
$$A_{(\alpha ,\beta )}:=(A, \tilde{\cdot}:=\cdot \circ (\alpha \otimes \beta), \tilde{*}:=* \circ (\alpha \otimes \beta),
\alpha , \beta )$$ 
is a BiHom-Novikov-Poisson algebra, called the Yau twist of $(A, \cdot , *)$. 
\end{proposition}
\begin{proof}
We only have to check (\ref{4.1}), (\ref{4.2}) and (\ref{new}) 
for $A_{(\alpha ,\beta )}$. For (\ref{4.1}) we compute:\\[2mm]
${\;\;\;\;\;\;}$
$(\beta(x)\tilde{\ast } \alpha(y))\tilde{\cdot }\beta(z)-\alpha\beta(x)\tilde{\ast }(\alpha(y)\tilde{\cdot }z)$
\begin{eqnarray*}
&=&(\alpha \beta(x)\ast  \alpha \beta (y))\tilde{\cdot }\beta(z)-\alpha\beta(x)\tilde{\ast }(\alpha ^2(y)\cdot \beta (z))\\
&=&(\alpha ^2\beta(x)\ast  \alpha ^2\beta (y))\cdot \beta ^2(z)-\alpha ^2\beta(x)\ast (\alpha ^2\beta (y)\cdot \beta ^2(z)), 
\end{eqnarray*}
and this expression is symmetric in $x$ and $y$ by (\ref{NP1}). Next, (\ref{4.2}) reads 
$$(x \tilde{\cdot }\beta (y))\tilde{\ast }\alpha\beta(z)=(x\tilde{\ast }\beta (z))\tilde{\cdot }\alpha\beta(y),$$
which is equivalent to 
$(\alpha (x) \cdot \beta ^2(y))\tilde{\ast }\alpha\beta(z)=(\alpha (x)\ast \beta ^2(z))\tilde{\cdot }\alpha\beta(y),$
which in turn is equivalent to 
$(\alpha ^2(x) \cdot \alpha \beta ^2(y))\ast \alpha\beta ^2(z)
=(\alpha ^2(x)\ast \alpha \beta ^2(z))\cdot \alpha\beta ^2(y),$
which is true by (\ref{NP2}). 

Finally, (\ref{new}) reads: 
$$\alpha (x)\tilde{\cdot} (y\tilde{\ast}z)=(x\tilde{\cdot}y)\tilde{\ast}\beta (z),$$
which is equivalent to $\alpha (x)\tilde{\cdot} (\alpha (y)\ast \beta (z))=(\alpha (x)\cdot \beta (y))\tilde{\ast}\beta (z),$
which in turn is equivalent to 
$\alpha ^2(x)\cdot (\alpha \beta (y)\ast \beta ^2(z))=(\alpha ^2(x)\cdot \alpha \beta (y))\ast \beta ^2(z),$
 which is true by (\ref{NP3}). 
\end{proof}

More generally, one can prove the following result:
\begin{proposition}\label{genYautwistNP}
Let $(A, \cdot, \ast, \alpha, \beta)$ be a BiHom-Novikov-Poisson algebra and let $\tilde{\alpha}, \tilde{\beta}: A\rightarrow A$ be  two  morphisms of BiHom-Novikov-Poisson algebras 
such that any two of the maps $\alpha, \beta, \tilde{\alpha}, \tilde{\beta}$ commute. Then
 $$A_{(\tilde{\alpha}, \tilde{\beta})}:=(A, \tilde{\cdot}:=\cdot \circ (\tilde{\alpha}\otimes \tilde{\beta}),\, \tilde{\ast}:=\ast \circ (\tilde{\alpha}\otimes \tilde{\beta}),\, \alpha\circ \tilde{\alpha}, \, \beta\circ \tilde{\beta})
 $$
 is also a BiHom-Novikov-Poisson algebra.
\end{proposition}
\begin{corollary}\label{Coro 4.2}
Let $(A, \cdot, \ast, \alpha , \beta)$ be a BiHom-Novikov-Poisson algebra. Then
 $$
 A^{n}:=(A, \cdot \circ (\alpha^{n}\otimes \beta^{n}), \ast \circ (\alpha^{n}\otimes \beta^{n}), \alpha^{n+1}, \beta^{n+1})
 $$ 
 is also a BiHom-Novikov-Poisson algebra for any $n\geq 0$.
\end{corollary}
\begin{proof}
Apply Proposition \ref{genYautwistNP} for $\tilde{\alpha }:=\alpha^{n}$ and
$\tilde{\beta }:=\beta^{n}$. 
\end{proof}

The following result is the special case of Corollary \ref{Coro 4.2} with $\ast=0$.
\begin{corollary}\label{Coro 4.3}
Let $(A, \cdot, \alpha , \beta)$ be a BiHom-commutative algebra. Then
 $$
 A^{n}:=(A, \cdot \circ (\alpha^{n}\otimes \beta^{n}), \alpha^{n+1}, \beta^{n+1})
 $$
 is also a BiHom-commutative algebra for any $n\geq 0$.
\end{corollary}

The following result is the special case of Corollary \ref{Coro 4.2} with $\cdot=0$.
\begin{corollary}\label{Coro 4.4}
Let $(A, \ast, \alpha , \beta)$ be a BiHom-Novikov algebra. Then
 $$
 A^{n}:=(A, \ast \circ (\alpha^{n}\otimes \beta^{n}), \alpha^{n+1}, \beta^{n+1})
 $$
 is also a BiHom-Novikov algebra for any $n\geq 0$.
\end{corollary}

The following consequence of Proposition  \ref{BHGelD} is useful for constructing examples of BiHom-Novikov-Poisson algebras.
\begin{corollary}\label{Coro 4.6}
Let $(A, \mu)$ be a commutative and associative algebra, let $\alpha, \beta: A\rightarrow A$ be two commuting algebra morphisms, 
and let $D: A\rightarrow A$ be a derivation such that $D\circ \alpha=\alpha\circ D$ and
$D\circ \beta=\beta\circ D$. Then $(A, \bullet, \ast, \alpha , \beta)$ is a BiHom-Novikov-Poisson algebra, where
\begin{eqnarray*}
&x\bullet y=\mu(\alpha(x)\otimes \beta(y)),\;\;\;\;\;\;\;x\ast y=\mu(\alpha(x)\otimes D(\beta(y))),\;\;\;\forall \;\;x, y\in A.
\end{eqnarray*}
\end{corollary}
\section{Infinitesimal BiHom-bialgebras}\label{sec4}
\setcounter{equation}{0} 

We introduce now the main concept of this paper. 
\begin{definition}
An infinitesimal BiHom-bialgebra is a 7-tuple $(A, \mu, \Delta , \alpha ,
\beta , \psi , \omega )$, with the properties that $(A, \mu , \alpha , \beta )$
is a BiHom-associative algebra, $(A, \Delta , \psi , \omega )$ is a
BiHom-coassociative coalgebra and the following conditions are satisfied,
for all $a, b\in A$:
\begin{eqnarray}
&& \Delta \circ \mu = (\mu \otimes \beta) \circ (\omega \otimes \Delta)+
(\alpha \otimes \mu)\circ (\Delta \otimes \psi),  \label{infin} \\
&& \alpha \circ \psi=\psi\circ \alpha,~~ \alpha \circ \omega=\omega \circ
\alpha, ~~ \beta \circ \psi=\psi \circ \beta, ~~ \beta \circ \omega=\omega
\circ \beta,  \label{infinb} \\
&& (\alpha \otimes \alpha)\circ \Delta=\Delta \circ \alpha, ~~(\beta \otimes
\beta)\circ \Delta=\Delta \circ \beta,  \label{infinbi} \\
&& \psi(a\cdot b)=\psi(a)\cdot \psi(b), ~~ \omega(a\cdot b)=\omega(a)\cdot \omega(b).
\label{infinbia}
\end{eqnarray}
\end{definition}

In terms of elements, 
the condition (\ref{infin}) can be rewritten as
\begin{equation}
\Delta (a\cdot b)=\omega (a)\cdot b_{1}\otimes \beta (b_{2})+\alpha (a_{1})\otimes
a_{2}\cdot \psi (b). \label{infinelements}
\end{equation}

The above axioms are justified by the following result: 
\begin{lemma}
\label{genGD} If $(A, \mu, \Delta , \alpha , \beta , \psi , \omega )$ is an
infinitesimal BiHom-bialgebra, 
then $\Delta: A\rightarrow A\otimes A$ is a
derivation (in the usual sense) of $A$ with values in the $A$-bimodule $%
(A\otimes A, \alpha \otimes \alpha, \beta \otimes \beta)$, in which the left
and the right $A$-actions are (for all $a,b, c\in A$)
\begin{eqnarray}
&& a \cdot (b\otimes c)=\omega(a)\cdot b\otimes \beta(c) ~~ and ~~ (b\otimes
c)\cdot a=\alpha(b) \otimes c\cdot \psi(a).  \label{infinMod}
\end{eqnarray}
\end{lemma}
\begin{proof}
We prove that $(A\otimes A, \alpha \otimes \alpha, \beta \otimes \beta)$
with the left and right actions defined above is an $A$-bimodule. For $a,
a^{\prime }, b, c\in A$, it is easy to check that $(\alpha \otimes \alpha)
(a \cdot (b\otimes c))= \alpha(a) \cdot ((\alpha \otimes \alpha)(b\otimes
c)) $ and $(\beta \otimes \beta) (a \cdot (b\otimes c))= \beta(a) \cdot
((\beta \otimes \beta)(b\otimes c))$. We compute:
\begin{eqnarray*}
\alpha(a)\cdot (a^{\prime }\cdot (b\otimes c)) &\overset{(\ref{infinMod})}{=}
&\alpha(a)\cdot (\omega(a^{\prime })\cdot b\otimes \beta(c)) \\
&\overset{(\ref{infinMod})}{=}&\omega \alpha(a)\cdot (\omega(a^{\prime })\cdot b)\otimes \beta ^2(c)\\
&\overset{(\ref{BHassoc})}{=}&\omega (a\cdot a^{\prime })\cdot \beta (b)\otimes \beta ^2(c)\\
&\overset{(\ref{infinMod})}{=}&(a\cdot a^{\prime })\cdot (\beta(b)\otimes \beta(c))\\
&=& (a\cdot a^{\prime })\cdot ((\beta \otimes \beta)(b\otimes c)).
\end{eqnarray*}
Thus $(A\otimes A, \alpha \otimes \alpha, \beta \otimes \beta)$ is a left $A$%
-module. Similarly one proves that it is also a right $A$-module. Finally,
we compute:
\begin{eqnarray*}
(a\cdot (b\otimes c))\cdot \beta(a^{\prime }) &\overset{(\ref{infinMod})}{=}%
& (\omega(a)\cdot b \otimes \beta(c))\cdot \beta(a^{\prime }) \\
&\overset{(\ref{infinMod})}{=}&\alpha\omega(a)\cdot \alpha(b)\otimes \beta (c)\cdot 
\psi\beta(a^{\prime }) \\
&\overset{(\ref{infinb})}{=}&\omega(\alpha (a))\cdot \alpha(b)\otimes \beta (c\cdot 
\psi(a^{\prime })) \\
&\overset{(\ref{infinMod})}{=}&\alpha(a)\cdot (\alpha(b)\otimes c\cdot 
\psi(a^{\prime })) \\
&\overset{(\ref{infinMod})}{=}& \alpha(a) \cdot ((b\otimes c)\cdot a^{\prime
}).
\end{eqnarray*}
Hence, $(A\otimes A, \alpha \otimes \alpha, \beta \otimes \beta)$ is indeed
an $A$-bimodule. From (\ref{infinelements}) and (\ref{infinMod}), we get $%
\Delta(a\cdot b)=a\cdot \Delta(b)+\Delta(a) \cdot b$, that is $\Delta $ is a
derivation.
\end{proof}

We show now how to obtain infinitesimal BiHom-bialgebras from infinitesimal bialgebras. 
\begin{proposition}
Let $(A, \mu , \Delta )$ be an infinitesimal bialgebra and let $\alpha , \beta ,
\psi , \omega :A\rightarrow A$ be morphisms of algebras and coalgebras such
that any two of them commute. Then $A_{(\alpha , \beta , \psi , \omega )}:=
(A, \mu _{(\alpha , \beta )}:=\mu \circ (\alpha \otimes \beta ), \Delta
_{(\psi , \omega )}:=(\omega \otimes \psi )\circ \Delta , \alpha , \beta ,
\psi , \omega )$ is an infinitesimal BiHom-bialgebra, called the Yau twist
of $(A, \mu , \Delta )$.
\end{proposition}

\begin{proof}
The fact that $(A, \mu _{(\alpha , \beta )}, \alpha , \beta
) $ is a BiHom-associative algebra and $(A, \Delta _{(\psi , \omega )}, \psi
, \omega )$ is a BiHom-coassociative coalgebra is known from \cite{gmmp}. It is easy to see that
conditions (\ref{infinb})-(\ref{infinbia}) are satisfied, so we only need to
prove (\ref{infin}). For simplicity, we denote $\mu _{(\alpha , \beta
)}(a\otimes b)=a*b=\alpha (a)\cdot \beta (b)$ and $\Delta _{(\psi , \omega
)}(a)=a_{[1]}\otimes a_{[2]}= \omega (a_1)\otimes \psi (a_2)$, for all $a,
b\in A$. We compute:\\[2mm]
${\;\;\;\;\;\;\;\;\;}$$\omega (a)*b_{[1]}\otimes \beta (b_{[2]})+\alpha
(a_{[1]})\otimes a_{[2]}*\psi (b)$
\begin{eqnarray*}
&=&\omega (a)*\omega (b_1)\otimes \beta \psi (b_2)+\alpha \omega
(a_1)\otimes \psi (a_2)*\psi (b) \\
&=&\alpha \omega (a)\cdot \beta \omega (b_1)\otimes \beta \psi (b_2)+\alpha \omega
(a_1)\otimes \alpha \psi (a_2) \cdot \beta \psi (b) \\
&=&\omega (\alpha (a)\cdot \beta (b_1))\otimes \beta \psi (b_2)+\alpha \omega
(a_1)\otimes \psi (\alpha (a_2)\cdot \beta (b)) \\
&=&(\omega \otimes \psi )(\alpha (a)\cdot \beta (b_1)\otimes \beta (b_2)+\alpha
(a_1)\otimes \alpha (a_2)\cdot \beta (b)) \\
&=&(\omega \otimes \psi )(\alpha (a)\cdot \beta (b)_1\otimes \beta (b)_2+\alpha
(a)_1\otimes \alpha (a)_2\cdot \beta (b)) \\
&=&(\omega \otimes \psi )(\Delta (\alpha (a)\cdot \beta (b)))=\Delta _{(\psi ,
\omega )}(a*b),
\end{eqnarray*}
finishing the proof.
\end{proof}

We recall the following well-known concept:
\begin{definition}
Let $A$ be an algebra, $\sigma ,\tau :A\rightarrow A$ algebra maps and 
$D:A\rightarrow A$ a linear map. We call $D$ a $(\tau , \sigma )$-derivation if 
$D(a\cdot b)=D(a)\cdot \tau (b)+\sigma (a)\cdot D(b)$, for all $a, b\in A$.
\end{definition}
\begin{remark}
Let $(A,\mu ,\Delta ,\alpha ,\beta ,\psi ,\omega )$ be an infinitesimal BiHom-bialgebra and define the linear map 
$D:A\rightarrow A$, $D:=\mu \circ \Delta $, that is $D(a)=a_1\cdot a_2$, for all $a\in A$. Then $D$ is a  
$(\beta \psi , \alpha \omega )$-derivation. Indeed, by using (\ref{infinelements}) and (\ref{BHassoc}) we can compute: 
\begin{eqnarray*}
D(a\cdot b)&=&\mu (\omega (a)\cdot b_1\otimes \beta (b_2)+\alpha (a_1)\otimes a_2\cdot \psi (b))\\
&=&(\omega (a)\cdot b_1)\cdot \beta (b_2)+\alpha (a_1)\cdot (a_2\cdot \psi (b))\\
&=&\alpha \omega (a)\cdot (b_1\cdot b_2)+(a_1\cdot a_2)\cdot \beta \psi (b)\\
&=& \alpha \omega (a)\cdot D(b)+D(a)\cdot \beta \psi (b), \;\;\;q.e.d.
\end{eqnarray*}
\end{remark}

We want to prove that one can associate a left BiHom-pre-Lie algebra to an infinitesimal BiHom-bialgebra 
$(A,\mu ,\Delta ,\alpha ,\beta ,\psi ,\omega )$.  
We need to guess the multiplication and the structure maps of the 
left BiHom-pre-Lie algebra we are looking for. We proceed as follows. We assume first that the infinitesimal BiHom-bialgebra 
$A$ is BiHom-commutative. In this case, one can check that the hypotheses of Proposition \ref{generalGD} are 
satisfied,  for $\gamma =\alpha ^2\beta \psi \omega $, $\lambda =\alpha \beta $, $\xi =id_A$ and the map $D$ 
defined by $D(a)=\alpha \beta \psi (a_1)\cdot \alpha ^2\omega (a_2)$, for all $a\in A$. So, by 
Proposition \ref{generalGD}, we obtain a left BiHom-pre-Lie algebra 
$(A, \ast , \alpha ^2\beta , \alpha ^{2}\beta ^{2}\psi \omega )$, 
where the multiplication $\ast $ is defined by $a\ast b=\alpha \beta (a)\cdot D(b)$, which, by using 
BiHom-associativity and BiHom-commutativity, may be expressed as follows: 
\begin{eqnarray*}
a\ast b&=&\alpha \beta (a)\cdot D(b)=\alpha \beta (a)\cdot (\alpha \beta \psi (b_1)\cdot \alpha ^2\omega (b_2))\\
&=&(\beta (a)\cdot \alpha \beta \psi (b_1))\cdot \alpha ^2\beta \omega (b_2)\\
&=&(\beta ^2\psi (b_1)\cdot \alpha (a))\cdot \alpha ^2\beta \omega (b_2)\\
&=&\alpha \beta ^2\psi (b_1)\cdot (\alpha (a)\cdot \alpha ^2\omega (b_2)).
\end{eqnarray*}
It turns out that this formula works also in the general case. Indeed, we have:
\begin{theorem}\label{infprelie}
Let $(A,\mu ,\Delta ,\alpha ,\beta ,\psi ,\omega )$ be an infinitesimal
BiHom-bialgebra.  
Define a new multiplication on $A$ by
\begin{eqnarray*}
a\ast b &=&\alpha \beta ^{2}\psi \left( b_{1}\right) \cdot \left[ \alpha
\left( a\right) \cdot \alpha ^{2}\omega \left( b_{2}\right) \right] 
=\left[ \beta ^{2}\psi \left( b_{1}\right) \cdot \alpha \left( a\right) %
\right] \cdot \alpha ^{2}\beta \omega \left( b_{2}\right).
\end{eqnarray*}%
Then $(A,\ast ,\alpha ^{2}\beta ,\alpha ^{2}\beta ^{2}\psi \omega )$ is a
left BiHom-pre-Lie algebra.
\end{theorem}

\begin{proof} It is easy to see that $\alpha ^2\beta (x\ast y)=\alpha ^2\beta (x)\ast \alpha ^2\beta (y)$ and 
$\alpha ^{2}\beta ^{2}\psi \omega (x\ast y)=\alpha ^{2}\beta ^{2}\psi \omega (x)\ast \alpha ^{2}\beta ^{2}\psi \omega (y)$, 
for all $x, y\in A$. 
So we only need to check that, for all $x, y, z\in A$, we have 
\begin{eqnarray*}
&&(\alpha ^{2}\beta ^{2}\psi \omega (x)\ast \alpha ^{2}\beta (y))\ast \alpha
^{2}\beta ^{2}\psi \omega (z)-\alpha ^{4}\beta ^{3}\psi \omega (x)\ast
(\alpha ^{2}\beta (y)\ast z) \\
&&\;\;\;\;\;\;\;\;=(\alpha ^{2}\beta ^{2}\psi \omega (y)\ast \alpha ^{2}\beta (x))\ast
\alpha ^{2}\beta ^{2}\psi \omega (z)-\alpha ^{4}\beta ^{3}\psi \omega
(y)\ast (\alpha ^{2}\beta (x)\ast z).
\end{eqnarray*}%
For $a, b, c \in A$, we compute:
\begin{eqnarray*}
\left( a\ast b\right) \ast c &=&\alpha \beta ^{2}\psi \left( c_{1}\right)
\cdot \left[ \alpha \left( a\ast b\right) \cdot \alpha ^{2}\omega \left(
c_{2}\right) \right] \\
&=&\alpha \beta ^{2}\psi \left( c_{1}\right) \cdot \left[ \alpha \left\{
\alpha \beta ^{2}\psi \left( b_{1}\right) \cdot \left[ \alpha \left(
a\right) \cdot \alpha ^{2}\omega \left( b_{2}\right) \right] \right\} \cdot
\alpha ^{2}\omega \left( c_{2}\right) \right] \\
&=&\alpha \beta ^{2}\psi \left( c_{1}\right) \cdot \left[ \left\{ \alpha
^{2}\beta ^{2}\psi \left( b_{1}\right) \cdot \left[ \alpha ^{2}\left(
a\right) \cdot \alpha ^{3}\omega \left( b_{2}\right) \right] \right\} \cdot
\alpha ^{2}\omega \left( c_{2}\right) \right], 
\end{eqnarray*}%
hence\\[2mm]
${\;\;\;}$
$(\alpha ^{2}\beta ^{2}\psi \omega (x)\ast \alpha ^{2}\beta (y))\ast \alpha
^{2}\beta ^{2}\psi \omega (z)$
\begin{eqnarray*}
&=&\alpha \beta ^{2}\psi \left( \alpha ^{2}\beta
^{2}\psi \omega (z_{1})\right) \cdot \\
&&\left[ \left\{ \alpha ^{2}\beta ^{2}\psi \left( \alpha ^{2}\beta
(y_{1})\right) \cdot \left[ \alpha ^{2}\left( \alpha ^{2}\beta ^{2}\psi
\omega (x)\right) \cdot \alpha ^{3}\omega \left( \alpha ^{2}\beta
(y_{2})\right) \right] \right\} \cdot \alpha ^{2}\omega \left( \alpha
^{2}\beta ^{2}\psi \omega (z_{2})\right) \right] \\
&=&\alpha ^{3}\beta ^{4}\psi ^{2}\omega (z_{1})\cdot \left\{ \left[ \alpha
^{4}\beta ^{3}\psi \left( y_{1}\right) \cdot \left( \alpha ^{4}\beta
^{2}\psi \omega (x)\cdot \alpha ^{5}\beta \omega \left( y_{2}\right) \right) %
\right] \cdot \alpha ^{4}\beta ^{2}\omega ^{2}\psi (z_{2})\right\} .
\end{eqnarray*}
For elements $s, t, w\in A$ we have, by (\ref{infinelements}), 
\begin{eqnarray*}
\Delta (t\cdot w) &=&\omega (t)\cdot w_{1}\otimes \beta (w_{2})+\alpha
(t_{1})\otimes t_{2}\cdot \psi (w),
\end{eqnarray*}
and so we can compute, by using again (\ref{infinelements}):
\begin{eqnarray*}
\Delta \left( s\cdot \left( t\cdot w\right) \right) &=&\omega (s)\cdot
\left( t\cdot w\right) _{1}\otimes \beta (\left( t\cdot w\right)
_{2})+\alpha (s_{1})\otimes s_{2}\cdot \psi (t\cdot w) \\
&=&\omega (s)\cdot \left( \omega (t)\cdot w_{1}\right) \otimes \beta (\beta
(w_{2}))+\omega (s)\cdot \left( \alpha (t_{1})\right) \otimes \beta
(t_{2}\cdot \psi (w))\\
&&+\alpha (s_{1})\otimes s_{2}\cdot \psi (t\cdot w) \\
&=&\omega (s)\cdot \left( \omega (t)\cdot w_{1}\right) \otimes \beta
^{2}(w_{2})+\omega (s)\cdot \alpha (t_{1})\otimes \beta (t_{2})\cdot \beta
\psi (w)\\
&&+\alpha (s_{1})\otimes s_{2}\cdot \psi (t\cdot w), 
\end{eqnarray*}
and we use this formula to compute:
\begin{eqnarray*}
\Delta \left( b\ast c\right) &=&\Delta \left( \alpha \beta ^{2}\psi \left(
c_{1}\right) \cdot \left[ \alpha \left( b\right) \cdot \alpha ^{2}\omega
\left( c_{2}\right) \right] \right) \\
&=&\omega (\alpha \beta ^{2}\psi \left(
c_{1}\right) )\cdot \left( \omega (\alpha \left( b\right) )\cdot \alpha
^{2}\omega \left( c_{2}\right) _{1}\right) \otimes \beta ^{2}(\alpha
^{2}\omega \left( c_{2}\right) _{2})\\
&&+\omega (\alpha \beta ^{2}\psi \left( c_{1}\right) )\cdot \alpha (\alpha
\left( b\right) _{1})\otimes \beta (\alpha \left( b\right)_{2})\cdot \beta
\psi (\alpha ^{2}\omega \left( c_{2}\right) )\\
&&+\alpha (\alpha \beta ^{2}\psi \left( c_{1}\right) _{1})\otimes \alpha
\beta ^{2}\psi \left( c_{1}\right) _{2}\cdot \psi (\alpha \left( b\right)
\cdot \alpha ^{2}\omega \left( c_{2}\right) )\\
&=&\alpha \beta ^{2}\psi \omega \left( c_{1}\right)\cdot \left[ \alpha
\omega \left( b\right) \cdot \alpha ^{2}\omega \left( c_{2_{1}}\right) %
\right] \otimes \alpha ^{2}\beta ^{2}\omega \left( c_{2_{2}}\right) \\
&&+\alpha \beta ^{2}\psi \omega \left( c_{1}\right) \cdot \alpha ^{2}\left(
b_{1}\right) \otimes \alpha \beta \left( b_{2}\right) \cdot \alpha ^{2}\beta
\psi \omega \left( c_{2}\right) \\
&&+\alpha ^{2}\beta ^{2}\psi \left( c_{1_{1}}\right) \otimes \alpha \beta
^{2}\psi \left( c_{1_{2}}\right) \cdot \left[ \alpha \psi \left( b\right)
\cdot \alpha ^{2}\psi \omega \left( c_{2}\right) \right]. 
\end{eqnarray*}
Finally, we use this formula to compute:\\[2mm]
${\;\;\;\;}$
$a\ast \left( b\ast c\right)$
\begin{eqnarray*}
&=&\alpha \beta ^{2}\psi (\left( b\ast c\right)
_{1})\cdot \left[ \alpha \left( a\right) \cdot \alpha ^{2}\omega (\left( b\ast
c\right) _{2})\right]\\
&=&\alpha \beta ^{2}\psi \left\{ \alpha \beta ^{2}\psi \omega \left(
c_{1}\right) \cdot \left[ \alpha \omega \left( b\right) \cdot \alpha
^{2}\omega \left( c_{2_{1}}\right) \right] \right\} \cdot \left[ \alpha
\left( a\right) \cdot \alpha ^{2}\omega \alpha ^{2}\beta ^{2}\omega \left(
c_{2_{2}}\right) \right] \\
&&+\alpha \beta ^{2}\psi \left[ \alpha \beta ^{2}\psi \omega \left(
c_{1}\right) \cdot \alpha ^{2}\left( b_{1}\right) \right] \cdot \left\{
\alpha \left( a\right) \cdot \alpha ^{2}\omega \left[ \alpha \beta \left(
b_{2}\right) \cdot \alpha ^{2}\beta \psi \omega \left( c_{2}\right) \right]
\right\} \\
&&+\alpha \beta ^{2}\psi \left( \alpha ^{2}\beta ^{2}\psi \left(
c_{1_{1}}\right) \right) \cdot \left[ \alpha \left( a\right) \cdot \alpha
^{2}\omega \left\{ \alpha \beta ^{2}\psi \left( c_{1_{2}}\right) \cdot 
\left[ \alpha \psi \left( b\right) \cdot \alpha ^{2}\psi \omega \left(
c_{2}\right) \right] \right\} \right] \\
&=&\left\{ \alpha \beta ^{2}\psi \alpha \beta ^{2}\psi \omega \left(
c_{1}\right) \cdot \left[ \alpha \beta ^{2}\psi \alpha \omega \left(
b\right) \cdot \alpha \beta ^{2}\psi \alpha ^{2}\omega \left(
c_{2_{1}}\right) \right] \right\} \cdot \left[ \alpha \left( a\right) \cdot
\alpha ^{4}\beta ^{2}\omega ^{2}\left( c_{2_{2}}\right) \right] \\
&&+\left[ \alpha \beta ^{2}\psi \alpha \beta ^{2}\psi \omega \left(
c_{1}\right) \cdot \alpha \beta ^{2}\psi \alpha ^{2}\left( b_1\right)
\right] \cdot \left\{ \alpha \left( a\right) \cdot \left[ \alpha
^{2}\omega \alpha \beta \left( b_{2}\right) \cdot \alpha ^{2}\omega \alpha
^{2}\beta \psi \omega \left( c_{2}\right) \right] \right\} \\
&&+\alpha ^{3}\beta ^{4}\psi ^{2}\left( c_{1_{1}}\right)
\cdot \left[ \alpha \left( a\right) \cdot \left\{ \alpha ^{2}\omega \alpha
\beta ^{2}\psi \left( c_{1_{2}}\right) \cdot \left[ \alpha ^{2}\omega
\alpha \psi \left( b\right) \cdot \alpha ^{2}\omega \alpha ^{2}\psi \omega
\left( c_{2}\right) \right] \right\} \right] \\
&=&\left\{ \alpha ^{2}\beta ^{4}\psi ^{2}\omega \left( c_{1}\right) \cdot
\left[ \alpha ^{2}\beta ^{2}\psi \omega \left( b\right) \cdot \alpha
^{3}\beta ^{2}\psi \omega \left( c_{2_{1}}\right) \right] \right\} \cdot %
\left[ \alpha \left( a\right) \cdot \alpha ^{4}\beta ^{2}\omega ^{2}\left(
c_{2_{2}}\right) \right] \\
&&+\left[ \alpha ^{2}\beta ^{4}\psi ^{2}\omega \left( c_{1}\right) \cdot
\alpha ^{3}\beta ^{2}\psi \left( b_{1}\right) \right] \cdot \left\{ \alpha
\left( a\right) \cdot \left[ \alpha ^{3}\omega \beta \left( b_{2}\right)
\cdot \alpha ^{4}\beta \psi \omega ^{2}\left( c_{2}\right) \right] \right\} 
\\
&&+\alpha ^{3}\beta ^{4}\psi ^{2}\left( c_{1_{1}}\right) 
\cdot \left[ \alpha \left( a\right) \cdot \left\{ \alpha ^{3}\beta ^{2}\psi
\omega \left( c_{1_{2}}\right) \cdot \left[ \alpha ^{3}\psi \omega \left(
b\right) \cdot \alpha ^{4}\psi \omega ^{2}\left( c_{2}\right) \right]
\right\} \right].
\end{eqnarray*}
By using this formula for $a\ast \left( b\ast c\right)$, we can compute: \\[2mm]
$\alpha ^{4}\beta ^{3}\psi \omega (x)\ast \left( \alpha ^{2}\beta (y)\ast
z\right)$
\begin{eqnarray*}
&=&\left\{ \alpha ^{2}\beta ^{4}\psi ^{2}\omega \left( z_{1}\right)
\cdot \left[ \alpha ^{2}\beta ^{2}\psi \omega \left( \alpha ^{2}\beta
(y)\right) \cdot \alpha ^{3}\beta ^{2}\psi \omega \left( z_{2_{1}}\right) %
\right] \right\} \cdot \left[ \alpha \left( \alpha ^{4}\beta ^{3}\psi \omega
(x)\right) \cdot \alpha ^{4}\beta ^{2}\omega ^{2}\left( z_{2_{2}}\right) %
\right] \\
&&+\left[ \alpha ^{2}\beta ^{4}\psi ^{2}\omega \left( z_{1}\right) \cdot
\alpha ^{3}\beta ^{2}\psi \left( \alpha ^{2}\beta (y_{1})\right) \right]
\cdot \left\{ \alpha \left( \alpha ^{4}\beta ^{3}\psi \omega (x)\right)
\cdot \left[ \alpha ^{3}\omega \beta \left( \alpha ^{2}\beta (y_{2})\right)
\cdot \alpha ^{4}\beta \psi \omega ^{2}\left( z_{2}\right) \right] \right\} 
\\
&&+\alpha ^{3}\beta ^{4}\psi ^{2}\left( z_{1_{1}}\right)\cdot
\left[ \alpha \left( \alpha ^{4}\beta ^{3}\psi \omega (x)\right) \cdot
\left\{ \alpha ^{3}\beta ^{2}\psi \omega \left( z_{1_{2}}\right) \cdot 
\left[ \alpha ^{3}\psi \omega \left( \alpha ^{2}\beta (y)\right) \cdot
\alpha ^{4}\psi \omega ^{2}\left( z_{2}\right) \right] \right\} \right] \\
&=&\left\{ \alpha ^{2}\beta ^{4}\psi ^{2}\omega \left( z_{1}\right) \cdot
\left[ \alpha ^{4}\beta ^{3}\psi \omega (y)\cdot \alpha ^{3}\beta ^{2}\psi
\omega \left( z_{2_{1}}\right) \right] \right\} \cdot \left[ \alpha
^{5}\beta ^{3}\psi \omega (x)\cdot \alpha ^{4}\beta ^{2}\omega ^{2}\left(
z_{2_{2}}\right) \right] \\
&&+\left[ \alpha ^{2}\beta ^{4}\psi ^{2}\omega \left( z_{1}\right) \cdot
\alpha ^{5}\beta ^{3}\psi (y_{1})\right] \cdot \left\{ \alpha ^{5}\beta
^{3}\psi \omega (x)\cdot \left[ \alpha ^{5}\beta ^{2}\omega \left(
y_{2}\right) \cdot \alpha ^{4}\beta \psi \omega ^{2}\left( z_{2}\right) %
\right] \right\} \\
&&+\alpha ^{3}\beta ^{4}\psi ^{2}\left( z_{1_{1}}\right) \cdot \left[ \alpha
^{5}\beta ^{3}\psi \omega (x)\cdot \left\{ \alpha ^{3}\beta ^{2}\psi \omega
\left( z_{1_{2}}\right) \cdot \left[ \alpha ^{5}\beta \psi \omega (y)\cdot
\alpha ^{4}\psi \omega ^{2}\left( z_{2}\right) \right] \right\} \right]. 
\end{eqnarray*}
We apply repeatedly the BiHom-associativity condition to compute the second term in this sum: \\[2mm]
${\;\;\;\;\;}$
$\left[ \alpha ^{2}\beta ^{4}\psi ^{2}\omega \left( z_{1}\right) \cdot
\alpha ^{5}\beta ^{3}\psi (y_{1})\right] \cdot \left\{ \alpha ^{5}\beta
^{3}\psi \omega (x)\cdot \left[ \alpha ^{5}\beta ^{2}\omega \left(
y_{2}\right) \cdot \alpha ^{4}\beta \psi \omega ^{2}\left( z_{2}\right) 
\right] \right\}$
\begin{eqnarray*}
&=&\alpha ^{3}\beta ^{4}\psi ^{2}\omega \left( z_{1}\right) \cdot \left[
\alpha ^{5}\beta ^{3}\psi (y_{1})\cdot \left\{ \alpha ^{5}\beta ^{2}\psi
\omega (x)\cdot \left[ \alpha ^{5}\beta \omega \left( y_{2}\right) \cdot
\alpha ^{4}\psi \omega ^{2}\left( z_{2}\right) \right] \right\} \right] \\
&=&\alpha ^{3}\beta ^{4}\psi ^{2}\omega \left( z_{1}\right) \cdot \left[
\alpha ^{5}\beta ^{3}\psi (y_{1})\cdot \left\{ \left[ \alpha ^{4}\beta
^{2}\psi \omega (x)\cdot \alpha ^{5}\beta \omega \left( y_{2}\right) \right]
\cdot \alpha ^{4}\beta \psi \omega ^{2}\left( z_{2}\right) \right\} \right]
\\
&=&\alpha ^{3}\beta ^{4}\psi ^{2}\omega \left( z_{1}\right) \cdot \left[
\left\{ \alpha ^{4}\beta ^{3}\psi (y_{1})\cdot \left[ \alpha ^{4}\beta
^{2}\psi \omega (x)\cdot \alpha ^{5}\beta \omega \left( y_{2}\right) \right]
\right\} \cdot \alpha ^{4}\beta ^{2}\psi \omega ^{2}\left( z_{2}\right) %
\right], 
\end{eqnarray*}
and this coincides with the expression we obtained before for 
$(\alpha ^{2}\beta ^{2}\psi \omega (x)\ast \alpha ^{2}\beta (y))\ast
\alpha ^{2}\beta ^{2}\psi \omega (z)$.

Thus the only thing that remains to be proved is that the expression 
\begin{gather*}
\left\{ \alpha ^{2}\beta ^{4}\psi ^{2}\omega \left( z_{1}\right) \cdot \left[
\alpha ^{4}\beta ^{3}\psi \omega (y)\cdot \alpha ^{3}\beta ^{2}\psi \omega
\left( z_{2_{1}}\right) \right] \right\} \cdot \left[ \alpha ^{5}\beta
^{3}\psi \omega (x)\cdot \alpha ^{4}\beta ^{2}\omega ^{2}\left(
z_{2_{2}}\right) \right] \\
\;\;\;\;\;\;\;\;\;\;+\alpha ^{3}\beta ^{4}\psi ^{2}\left( z_{1_{1}}\right) \cdot \left[ \alpha
^{5}\beta ^{3}\psi \omega (x)\cdot \left\{ \alpha ^{3}\beta ^{2}\psi \omega
\left( z_{1_{2}}\right) \cdot \left[ \alpha ^{5}\beta \psi \omega (y)\cdot
\alpha ^{4}\psi \omega ^{2}\left( z_{2}\right) \right] \right\} \right]
\end{gather*}
is symmetric in $x$ and $y$. We compute, by applying both the BiHom-associativity condition and the BiHom-coassociativity 
condition $(\Delta \otimes \psi )\circ \Delta =(\omega \otimes \Delta )\circ \Delta $: \\[2mm]
${\;\;\;\;\;}$
$\left\{ \alpha ^{2}\beta ^{4}\psi ^{2}\omega \left( z_{1}\right) \cdot \left[
\alpha ^{4}\beta ^{3}\psi \omega (y)\cdot \alpha ^{3}\beta ^{2}\psi \omega
\left( z_{2_{1}}\right) \right] \right\} \cdot \left[ \alpha ^{5}\beta
^{3}\psi \omega (x)\cdot \alpha ^{4}\beta ^{2}\omega ^{2}\left(
z_{2_{2}}\right) \right]$
\begin{eqnarray*}
&=&\alpha ^{3}\beta ^{4}\psi ^{2}\omega \left( z_{1}\right) \cdot \left\{
\left[ \alpha ^{4}\beta ^{3}\psi \omega (y)\cdot \alpha ^{3}\beta ^{2}\psi
\omega \left( z_{2_{1}}\right) \right] \cdot \left[ \alpha ^{5}\beta
^{2}\psi \omega (x)\cdot \alpha ^{4}\beta \omega ^{2}\left( z_{2_{2}}\right) %
\right] \right\} \\
&=&\alpha ^{3}\beta ^{4}\psi ^{2}\omega \left( z_{1}\right) \cdot \left[
\alpha ^{5}\beta ^{3}\psi \omega (y)\cdot \left\{ \alpha ^{3}\beta ^{2}\psi
\omega \left( z_{2_{1}}\right) \cdot \left[ \alpha ^{5}\beta \psi \omega
(x)\cdot \alpha ^{4}\omega ^{2}\left( z_{2_{2}}\right) \right] \right\} %
\right] \\
&=&\alpha ^{3}\beta ^{4}\psi ^{2}\left( z_{1_{1}}\right) \cdot \left[ \alpha
^{5}\beta ^{3}\psi \omega (y)\cdot \left\{ \alpha ^{3}\beta ^{2}\psi \omega
\left( z_{1_{2}}\right) \cdot \left[ \alpha ^{5}\beta \psi \omega (x)\cdot
\alpha ^{4}\psi \omega ^{2}\left( z_{2}\right) \right] \right\} \right],
\end{eqnarray*}
so that%
\begin{gather*}
\left\{ \alpha ^{2}\beta ^{4}\psi ^{2}\omega \left( z_{1}\right) \cdot \left[
\alpha ^{4}\beta ^{3}\psi \omega (y)\cdot \alpha ^{3}\beta ^{2}\psi \omega
\left( z_{2_{1}}\right) \right] \right\} \cdot \left[ \alpha ^{5}\beta
^{3}\psi \omega (x)\cdot \alpha ^{4}\beta ^{2}\omega ^{2}\left(
z_{2_{2}}\right) \right] \\
\;\;\;\;\;\;\;\;\;\;+\alpha ^{3}\beta ^{4}\psi ^{2}\left( z_{1_{1}}\right) \cdot \left[ \alpha
^{5}\beta ^{3}\psi \omega (x)\cdot \left\{ \alpha ^{3}\beta ^{2}\psi \omega
\left( z_{1_{2}}\right) \cdot \left[ \alpha ^{5}\beta \psi \omega (y)\cdot
\alpha ^{4}\psi \omega ^{2}\left( z_{2}\right) \right] \right\} \right] \\
=\alpha ^{3}\beta ^{4}\psi ^{2}\left( z_{1_{1}}\right) \cdot \left[ \alpha
^{5}\beta ^{3}\psi \omega (y)\cdot \left\{ \alpha ^{3}\beta ^{2}\psi \omega
\left( z_{1_{2}}\right) \cdot \left[ \alpha ^{5}\beta \psi \omega (x)\cdot
\alpha ^{4}\psi \omega ^{2}\left( z_{2}\right) \right] \right\} \right] \\
\;\;\;\;\;\;\;\;\;\;+\alpha ^{3}\beta ^{4}\psi ^{2}\left( z_{1_{1}}\right) \cdot \left[ \alpha
^{5}\beta ^{3}\psi \omega (x)\cdot \left\{ \alpha ^{3}\beta ^{2}\psi \omega
\left( z_{1_{2}}\right) \cdot \left[ \alpha ^{5}\beta \psi \omega (y)\cdot
\alpha ^{4}\psi \omega ^{2}\left( z_{2}\right) \right] \right\} \right],
\end{gather*}
which is obviously symmetric in $x$ and $y$.
\end{proof}
\section{Quasitriangular infinitesimal BiHom-bialgebras}\label{sec5}
\setcounter{equation}{0} 
We begin this section with a result of independent interest. 
\begin{proposition}
Let $(A,\mu ,\alpha ,\beta )$ be a BiHom-associative algebra 
and let $n\geq 2$ be a natural number. Consider the following left and right actions of $A$ on $A^{\otimes n}$, for 
all $a, b_1, b_2, \cdots , b_n\in A$:
\begin{eqnarray*}
&&a\bullet (b_1\otimes b_2\otimes \cdots \otimes b_n)=\alpha (a)\cdot b_1\otimes \beta (b_2)\otimes 
\cdots \otimes \beta (b_n), \\
&&(b_1\otimes b_2\otimes \cdots \otimes b_n)\bullet a=\alpha (b_1)\otimes \cdots \otimes \alpha (b_{n-1})\otimes 
b_n\cdot \beta (a).
\end{eqnarray*}
Then with these actions $(A^{\otimes n}, \alpha ^{\otimes n}, \beta ^{\otimes n})$ is an $A$-bimodule.
\end{proposition}
\begin{proof}
A straightforward computation left to the reader.
\end{proof}

We will be interested in the case $n=3$, so we have the following actions of $A$ on $A\otimes A\otimes A$:
\begin{eqnarray}
&&a\bullet (x\otimes y\otimes z)=\alpha (a)\cdot x\otimes \beta (y)\otimes \beta (z), \label{action1}\\
&&(x\otimes y\otimes z)\bullet a=\alpha (x)\otimes \alpha (y)\otimes 
z\cdot \beta (a).\label{action2}
\end{eqnarray}
\begin{proposition}
Let $(A,\mu ,\alpha ,\beta )$ be a BiHom-associative algebra and let $r\in A\otimes A$ be an element, 
with notation $r=\sum _ix_i\otimes y_i$, such that 
$(\alpha \otimes \alpha )(r)=r=(\beta \otimes \beta )(r)$. Define the linear map
\begin{eqnarray*}
&&\Delta _r:A\rightarrow A\otimes A, \;\;\;\Delta _r(a)=\sum _i\alpha (x_i)\otimes y_i\cdot a-\sum _ia\cdot x_i
\otimes \beta (y_i), \;\;\;\;\;\forall \;\;a\in A. 
\end{eqnarray*}
Then we have $(\alpha \otimes \alpha)\circ \Delta _r=\Delta _r\circ \alpha$, $(\beta \otimes
\beta)\circ \Delta _r=\Delta _r\circ \beta $ and, 
if we denote as usual $\Delta _r(a)=a_1\otimes a_2$, for $a\in A$, the following identity holds:
\begin{eqnarray*}
&&\Delta _r(a\cdot b)=\alpha (a)\cdot b_{1}\otimes \beta (b_{2})+\alpha (a_{1})\otimes
a_{2}\cdot \beta (b), \;\;\;\;\;\forall \;\;a, b\in A.
\end{eqnarray*}
\end{proposition}
\begin{proof} The fact that $(\alpha \otimes \alpha)\circ \Delta _r=\Delta _r\circ \alpha$ and $(\beta \otimes
\beta)\circ \Delta _r=\Delta _r\circ \beta $ follows immediately from the condition 
$(\alpha \otimes \alpha )(r)=r=(\beta \otimes \beta )(r)$. Now 
we compute:\\[2mm]
${\;\;\;\;\;}$
$\alpha (a)\cdot b_{1}\otimes \beta (b_{2})+\alpha (a_{1})\otimes
a_{2}\cdot \beta (b)$
\begin{eqnarray*}
&=&\sum _i \alpha (a)\cdot \alpha (x_i)\otimes \beta (y_i\cdot b)-\sum _i \alpha (a)\cdot (b\cdot x_i)\otimes \beta (\beta (y_i)) \\
&&+\sum _i \alpha (\alpha (x_i))\otimes (y_i\cdot a)\cdot \beta (b)-\sum _i \alpha (a\cdot x_i)\otimes \beta (y_i)\cdot \beta (b)\\
&=&\sum \alpha ^2(x_i)\otimes \alpha (y_i)\cdot (a\cdot b)-\sum _i (a\cdot b)\cdot \beta (x_i)\otimes \beta ^2(y_i)\\
&=&\sum \alpha (x_i)\otimes y_i\cdot (a\cdot b)-\sum _i (a\cdot b)\cdot x_i\otimes \beta (y_i), 
\end{eqnarray*}
where for the last equality we used the fact that $(\alpha \otimes \alpha )(r)=r=(\beta \otimes \beta )(r)$. The expression 
we obtained is exactly $\Delta _r(a\cdot b)$. 
\end{proof}

We recall the following notation introduced in \cite{lmmp3}:
\begin{definition}
Let $(A,\mu ,\alpha ,\beta )$ be a BiHom-associative algebra and let $r=\sum _ix_i\otimes y_i\in A\otimes A$ be such that 
$(\alpha \otimes \alpha )(r)=r=(\beta \otimes \beta )(r)$. We define the following elements in $A\otimes A\otimes A$:
\begin{eqnarray*}
&&r_{12}r_{23}=\sum _{i, j}\alpha (x_i)\otimes y_i\cdot x_j\otimes \beta (y_j), \\
&&r_{13}r_{12}=\sum _{i, j}x_i\cdot x_j\otimes \beta (y_j)\otimes \beta (y_i), \\
&&r_{23}r_{13}=\sum _{i, j}\alpha (x_i)\otimes \alpha (x_j)\otimes y_j\cdot y_i, \\ 
&&A(r)=r_{13}r_{12}-
r_{12}r_{23}+r_{23}r_{13}.
\end{eqnarray*}
\end{definition}
\begin{proposition}
Let $(A,\mu ,\alpha ,\beta )$ be a BiHom-associative algebra and let 
$r=\sum _ix_i\otimes y_i\in A\otimes A$ be such that $(\alpha \otimes \alpha )(r)=r=(\beta \otimes \beta )(r)$. 
Then we have that $(\Delta _r \otimes \beta )\circ \Delta _r=
(\alpha \otimes \Delta _r)\circ \Delta _r$ if and only if $a\bullet A(r)=A(r)\bullet a$ for all $a\in A$, where the actions $\bullet $ 
are defined by (\ref{action1}) and (\ref{action2}). 
\end{proposition}
\begin{proof}
We compute, for $a\in A$:\\[2mm]
${\;\;\;\;\;}$$(\Delta _r \otimes \beta )(\Delta _r(a))$
\begin{eqnarray*}
&=&(\Delta _r \otimes \beta )(\sum _i\alpha (x_i)\otimes y_i\cdot a-\sum _ia\cdot x_i
\otimes \beta (y_i))\\
&=&\sum _i\Delta _r(\alpha (x_i))\otimes \beta (y_i)\cdot \beta (a)-\sum _i\Delta _r(a\cdot x_i)\otimes \beta ^2(y_i)\\
&=&\sum _{i, j}\alpha (x_j)\otimes y_j\cdot \alpha (x_i)\otimes \beta (y_i)\cdot \beta (a)-
\sum _{i, j}\alpha (x_i)\cdot x_j\otimes \beta (y_j)\otimes \beta (y_i)\cdot \beta (a)\\
&&-\sum _{i, j}\alpha (x_j)\otimes y_j\cdot (a\cdot x_i)\otimes \beta ^2(y_i)+
\sum _{i, j}(a\cdot x_i)\cdot x_j\otimes \beta (y_j)\otimes \beta ^2(y_i)\\
&\overset{(\beta \otimes \beta )(r)=r}{=}&\sum _{i, j}\alpha (x_j)\otimes y_j\cdot \alpha (x_i)\otimes \beta (y_i)\cdot \beta (a)-
\sum _{i, j}\alpha (x_i)\cdot x_j\otimes \beta (y_j)\otimes \beta (y_i)\cdot \beta (a)\\
&&-\sum _{i, j}\alpha (x_j)\otimes y_j\cdot (a\cdot x_i)\otimes \beta ^2(y_i)+
\sum _{i, j}(a\cdot x_i)\cdot \beta (x_j)\otimes \beta ^2(y_j)\otimes \beta ^2(y_i)\\
&\overset{(\alpha \otimes \alpha )(r)=r}{=}&\sum _{i, j}\alpha (x_j)\otimes y_j\cdot \alpha (x_i)\otimes \beta (y_i)\cdot \beta (a)-
\sum _{i, j}\alpha (x_i)\cdot x_j\otimes \beta (y_j)\otimes \beta (y_i)\cdot \beta (a)\\
&&-\sum _{i, j}\alpha ^2(x_j)\otimes \alpha (y_j)\cdot (a\cdot x_i)\otimes \beta ^2(y_i)+
\sum _{i, j}(a\cdot x_i)\cdot \beta (x_j)\otimes \beta ^2(y_j)\otimes \beta ^2(y_i)\\
&=&\sum _{i, j}\alpha (x_j)\otimes y_j\cdot \alpha (x_i)\otimes \beta (y_i)\cdot \beta (a)-
\sum _{i, j}\alpha (x_i)\cdot x_j\otimes \beta (y_j)\otimes \beta (y_i)\cdot \beta (a)\\
&&-\sum _{i, j}\alpha ^2(x_j)\otimes (y_j\cdot a)\cdot \beta (x_i)\otimes \beta ^2(y_i)+
\sum _{i, j}\alpha (a)\cdot (x_i\cdot x_j)\otimes \beta ^2(y_j)\otimes \beta ^2(y_i),
\end{eqnarray*}
where for the last equality we applied two times the BiHom-associativity condition, and \\[2mm]
${\;\;\;\;\;}$$(\alpha \otimes \Delta _r)(\Delta _r(a))$
\begin{eqnarray*}
&=&(\alpha \otimes \Delta _r)(\sum _i\alpha (x_i)\otimes y_i\cdot a-\sum _ia\cdot x_i
\otimes \beta (y_i))\\
&=&\sum _i\alpha ^2(x_i)\otimes \Delta _r(y_i\cdot a)-\sum _i\alpha (a)\cdot \alpha (x_i)\otimes 
\Delta _r(\beta (y_i))\\
&=&\sum _{i, j}\alpha ^2(x_i)\otimes \alpha (x_j)\otimes y_j\cdot (y_i\cdot a)-
\sum _{i, j}\alpha ^2(x_i)\otimes (y_i\cdot a)\cdot x_j\otimes \beta (y_j)\\
&&-\sum _{i, j}\alpha (a)\cdot \alpha (x_i)\otimes \alpha (x_j)\otimes y_j\cdot \beta (y_i)+
\sum _{i, j}\alpha (a)\cdot \alpha (x_i)\otimes \beta (y_i)\cdot x_j\otimes \beta (y_j)\\
&\overset{(\alpha \otimes \alpha )(r)=r}{=}&\sum _{i, j}\alpha ^2(x_i)\otimes \alpha ^2(x_j)
\otimes \alpha (y_j)\cdot (y_i\cdot a)-
\sum _{i, j}\alpha ^2(x_i)\otimes (y_i\cdot a)\cdot x_j\otimes \beta (y_j)\\
&&-\sum _{i, j}\alpha (a)\cdot \alpha (x_i)\otimes \alpha (x_j)\otimes y_j\cdot \beta (y_i)+
\sum _{i, j}\alpha (a)\cdot \alpha (x_i)\otimes \beta (y_i)\cdot x_j\otimes \beta (y_j)\\
&\overset{(\beta \otimes \beta )(r)=r}{=}&\sum _{i, j}\alpha ^2(x_i)\otimes \alpha ^2(x_j)
\otimes \alpha (y_j)\cdot (y_i\cdot a)-
\sum _{i, j}\alpha ^2(x_i)\otimes (y_i\cdot a)\cdot \beta (x_j)\otimes \beta ^2(y_j)\\
&&-\sum _{i, j}\alpha (a)\cdot \alpha (x_i)\otimes \alpha (x_j)\otimes y_j\cdot \beta (y_i)+
\sum _{i, j}\alpha (a)\cdot \alpha (x_i)\otimes \beta (y_i)\cdot x_j\otimes \beta (y_j). 
\end{eqnarray*}
So, by using these formulae, we have $(\Delta _r \otimes \beta )(\Delta _r(a))=(\alpha \otimes \Delta _r)(\Delta _r(a))$ 
if and only if 
\begin{eqnarray*}
&&\sum _{i, j}\alpha (x_j)\otimes y_j\cdot \alpha (x_i)\otimes \beta (y_i)\cdot \beta (a)-
\sum _{i, j}\alpha (x_i)\cdot x_j\otimes \beta (y_j)\otimes \beta (y_i)\cdot \beta (a)\\
&&\;\;\;\;\;\;\;\;+\sum _{i, j}\alpha (a)\cdot (x_i\cdot x_j)\otimes \beta ^2(y_j)\otimes \beta ^2(y_i)\\
&&=\sum _{i, j}\alpha ^2(x_i)\otimes \alpha ^2(x_j)
\otimes \alpha (y_j)\cdot (y_i\cdot a)
-\sum _{i, j}\alpha (a)\cdot \alpha (x_i)\otimes \alpha (x_j)\otimes y_j\cdot \beta (y_i)\\
&&+
\sum _{i, j}\alpha (a)\cdot \alpha (x_i)\otimes \beta (y_i)\cdot x_j\otimes \beta (y_j), 
\end{eqnarray*}
which, by using one more time the BiHom-associativity condition, the fact that  $(\alpha \otimes \alpha )(r)=(\beta \otimes \beta )(r)=r$ 
and separating the terms, may be rewritten as 
\begin{eqnarray*}
&&\sum _{i, j}\alpha (a)\cdot (x_i\cdot x_j)\otimes \beta ^2(y_j)\otimes \beta ^2(y_i)
+\sum _{i, j}\alpha (a)\cdot \alpha (x_i)\otimes \alpha \beta (x_j)\otimes \beta (y_j\cdot y_i)\\
&&\;\;\;\;\;-\sum _{i, j}\alpha (a)\cdot \alpha (x_i)\otimes \beta (y_i\cdot x_j)\otimes \beta ^2(y_j)\\
&=&\sum _{i, j}\alpha ^2(x_i)\otimes \alpha ^2(x_j)\otimes (y_j\cdot y_i)\cdot \beta (a)+
\sum _{i, j}\alpha (x_i\cdot x_j)\otimes \alpha \beta (y_j)\otimes \beta (y_i)\cdot \beta (a)\\
&&\;\;\;\;\;-\sum _{i, j}\alpha ^2(x_j)\otimes \alpha (y_j\cdot x_i)\otimes \beta (y_i)\cdot \beta (a),
\end{eqnarray*}
and this is exactly the condition $a\bullet A(r)=A(r)\bullet a$. 
\end{proof}

As a consequence of the previous results, we obtain:
\begin{proposition}\label{propcobound}
Let $(A,\mu ,\alpha ,\beta )$ be a BiHom-associative algebra 
and let $r=\sum _ix_i\otimes y_i\in A\otimes A$ be such that $(\alpha \otimes \alpha )(r)=r=(\beta \otimes \beta )(r)$ 
and $a\bullet A(r)=A(r)\bullet a$, for all $a\in A$. Then, if we define the linear map 
\begin{eqnarray*}
&&\Delta _r:A\rightarrow A\otimes A, \;\;\;\Delta _r(a)=\sum _i\alpha (x_i)\otimes y_i\cdot a-\sum _ia\cdot x_i
\otimes \beta (y_i), \;\;\;\;\;\forall \;\;a\in A,  
\end{eqnarray*}
then $(A, \mu , \Delta _r, \alpha , \beta , \psi =\beta , \omega =\alpha )$ is an infinitesimal BiHom-bialgebra. 
\end{proposition}
\begin{definition}
An infinitesimal BiHom-bialgebra as in Proposition \ref{propcobound} is called a coboundary infinitesimal BiHom-bialgebra. 
\end{definition}
\begin{definition}
A coboundary infinitesimal BiHom-bialgebra for which $A(r)=0$ is called a quasitriangular infinitesimal BiHom-bialgebra. 
\end{definition}

These concepts extend Aguiar's classical concepts of coboundary and quasitriangular infinitesimal bialgebras (see 
\cite{aguiarcontemp}, \cite{aguiarjalgebra}), as well as their Hom-versions introduced by Yau in \cite{yauinf}. 

We recall from \cite{lmmp3} that 
the equation $A(r)=0$, i.e. 
\begin{eqnarray}
&&r_{13}r_{12}-
r_{12}r_{23}+r_{23}r_{13}=0, \label{BHAYBE1}
\end{eqnarray}
or, more concretely, 
\begin{eqnarray}
&&\sum _{i, j}\alpha (x_i)\otimes y_i\cdot x_j\otimes \beta (y_j)=\sum _{i, j}\{x_i\cdot x_j\otimes \beta (y_j)\otimes \beta (y_i)+
\alpha (x_i)\otimes \alpha (x_j)\otimes y_j\cdot y_i\}, \label{BHAYBE2}
\end{eqnarray}
is called the associative BiHom-Yang-Baxter equation. 

We have the following characterization of quasitriangular infinitesimal BiHom-bialgebras (extending, up to a sign convention, 
the ones for quasitriangular infinitesimal bialgebras in \cite{aguiarcontemp} and for quasitriangular 
infinitesimal Hom-bialgebras in \cite{yauinf}). 
\begin{proposition}
Let $(A, \mu , \Delta _r, \alpha , \beta , \psi =\beta , \omega =\alpha )$ be a quasitriangular infinitesimal BiHom-bialgebra 
and denote $\Delta =\Delta _r$ ($r=\sum _ix_i\otimes y_i\in A\otimes A$). Then: \\
(i) $\Delta (a)=\sum _i\alpha (x_i)\otimes y_i\cdot a-\sum _ia\cdot x_i\otimes \beta (y_i)$, for all $a\in A$; \\
(ii) $(\Delta \otimes \beta )(r)=r_{23}r_{13}$; \\
(iii) $(\alpha \otimes \Delta )(r)=-r_{13}r_{12}$. 

Conversely, if an infinitesimal BiHom-bialgebra $(A, \mu , \Delta , \alpha , \beta , \psi =\beta , \omega =\alpha )$ 
satisfies the relations (i), (ii) and (iii) for some $r=\sum _ix_i\otimes y_i\in A\otimes A$ with 
$(\alpha \otimes \alpha )(r)=r=(\beta \otimes \beta )(r)$, 
then $\Delta =\Delta _r$ and $(A, \mu , \Delta _r, \alpha , \beta , \psi =\beta , \omega =\alpha )$ is a 
quasitriangular infinitesimal BiHom-bialgebra. 
\end{proposition}
\begin{proof}
We prove first the direct implication. (i) is just the definition of $\Delta _r$. We prove (ii): 
\begin{eqnarray*}
(\Delta \otimes \beta )(r)&=&\sum _i \Delta (x_i)\otimes \beta (y_i)\\
&=&\sum _{i, j} (\alpha (x_j)\otimes y_j\cdot x_i-x_i\cdot x_j\otimes \beta (y_j))\otimes \beta (y_i)\\
&=&\sum _{i, j}\alpha (x_j)\otimes y_j\cdot x_i\otimes \beta (y_i)-\sum _{i, j}x_i\cdot x_j\otimes \beta (y_j)\otimes \beta (y_i)\\
&=&r_{12}r_{23}-r_{13}r_{12}\\
&\overset{(\ref{BHAYBE1})}{=}&r_{23}r_{13}.
\end{eqnarray*}
The proof of (iii) is similar and left to the reader. 

For the converse, it is enough if we know (i) and (ii) (or (i) and (iii)), because (i) says anyway that $\Delta $ is of the form 
$\Delta _r$,with $r=\sum _ix_i\otimes y_i\in A\otimes A$, and the computation performed for (ii) up to the 
last step can be done also now, and we get  $(\Delta \otimes \beta )(r)=r_{12}r_{23}-r_{13}r_{12}$, and so (ii) implies 
$r_{12}r_{23}-r_{13}r_{12}=r_{23}r_{13}$, that is (\ref{BHAYBE1}) holds, i.e. 
$(A, \mu , \Delta _r, \alpha , \beta , \psi =\beta , \omega =\alpha )$ is a quasitriangular infinitesimal BiHom-bialgebra. 
\end{proof}

We recall some other facts from \cite{lmmp3}. 
\begin{definition} (\cite{lmmp3}) 
Let $(A,\mu ,\alpha ,\beta )$ be a BiHom-associative algebra 
and let $R:A\rightarrow A$ be a linear map commuting with $\alpha $ and $\beta $. 
We call $R$ an $\alpha \beta $-Rota-Baxter operator if 
\begin{eqnarray}
&&R(\alpha \beta (a))\cdot R(\alpha \beta (b))=R(\alpha \beta (a)\cdot R(b)+R(a)\cdot \alpha \beta (b)), \;\;\;\forall \;\;a, b\in A. \label{generRB}
\end{eqnarray}
\end{definition}
\begin{proposition}\label{alphabetaRB} (\cite{lmmp3})
Let $(A,\mu ,\alpha ,\beta )$ be a BiHom-associative algebra 
and let $R:A\rightarrow A$ be an $\alpha \beta $-Rota-Baxter operator. Let $\eta :A\rightarrow A$ be a linear map, 
commuting with $\alpha , \beta , R$ and having the property that $\eta (x\cdot y)=\eta (x)\cdot \eta (y)$, for all $x, y\in A$. 
Define new operations on $A$ by 
\begin{eqnarray*}
&&x\prec y=\alpha \beta (x)\cdot R\eta (y) \;\;\;and\;\;\; x\succ y=R(x)\cdot \alpha \beta \eta (y), 
\end{eqnarray*} 
for all $x, y\in A$. Then $(A, \prec , \succ , \alpha ^2\beta , \alpha \beta ^2\eta )$ is a BiHom-dendriform algebra. 
\end{proposition}

Consequently, by using Proposition \ref{BHdendpreLie}, we obtain: 
\begin{corollary} \label{corolalphabetaRB}
In the hypotheses of Proposition \ref{alphabetaRB} and assuming that $\alpha, \beta , \eta $ are bijective, 
if we define a new operation on $A$ by 
\begin{eqnarray*}
&&x\star y=x\succ y-(\alpha ^{-1}\beta \eta (y))\prec (\alpha \beta ^{-1}\eta ^{-1}(x))=
R(x)\cdot \alpha \beta \eta (y)-\beta ^2\eta (y)\cdot R\alpha \beta ^{-1}(x),  
\end{eqnarray*}
for all $x, y\in A$, then $(A, \star , \alpha ^2\beta , \alpha \beta ^2\eta )$  
is a left BiHom-pre-Lie algebra.
\end{corollary}

\begin{theorem}\label{ABRB} (\cite{lmmp3})
Let $(A,\mu ,\alpha ,\beta )$ be a BiHom-associative algebra 
and let $r=\sum _ix_i\otimes y_i\in A\otimes A$ be such that $(\alpha \otimes \alpha )(r)=r=(\beta \otimes \beta )(r)$ 
and $r$ is a solution of the associative BiHom-Yang-Baxter equation. Define the linear map 
\begin{eqnarray*}
&&R:A\rightarrow A, \;\;\;R(a)=\sum _i\alpha \beta ^3(x_i)\cdot (a\cdot \alpha ^3(y_i))=
\sum _i(\beta ^3(x_i)\cdot a)\cdot \alpha ^3\beta (y_i), \;\;\;\;\;\forall \;\;a\in A. 
\end{eqnarray*}
Then $R$ is an $\alpha \beta $-Rota-Baxter operator. 
\end{theorem}

Let now $(A, \mu , \Delta _r, \alpha , \beta , \psi =\beta , \omega =\alpha )$ be a quasitriangular infinitesimal BiHom-bialgebra. 
By applying Theorem \ref{infprelie}, we obtain a BiHom-pre-Lie algebra $(A, \ast , \alpha ^2\beta , \alpha ^3\beta ^3)$, 
where the multiplication $\ast $ becomes:
\begin{eqnarray*}
a\ast b&=&\alpha \beta ^3(b_1)\cdot (\alpha (a)\cdot \alpha ^3(b_2))\\
&=&\sum _i\alpha \beta ^3(\alpha (x_i))\cdot (\alpha (a)\cdot \alpha ^3(y_i\cdot b))-\sum _i\alpha \beta ^3(b\cdot x_i)
\cdot (\alpha (a)\cdot \alpha ^3(\beta (y_i)))\\
&=&\sum _i\alpha ^2\beta ^3(x_i)\cdot (\alpha (a)\cdot \alpha ^3(y_i\cdot b))-\sum _i\alpha \beta ^3(b\cdot x_i)
\cdot (\alpha (a)\cdot \alpha ^3\beta (y_i)).
\end{eqnarray*}

On the other hand, to $(A, \mu , \Delta _r, \alpha , \beta , \psi =\beta , \omega =\alpha )$ we can associate 
the $\alpha \beta $-Rota-Baxter operator $R$ as in Theorem \ref{ABRB}, from which, by using 
Proposition \ref{alphabetaRB}, and choosing $\eta =\alpha ^2\beta $ there, we obtain a BiHom-dendriform algebra 
$(A, \prec , \succ , \alpha ^2\beta , \alpha ^3\beta ^3 )$. Assume now that moreover $\alpha $ and $\beta $ are 
bijective. Then, by Corollary  \ref{corolalphabetaRB}, we obtain a left BiHom-pre-Lie algebra 
$(A, \star , \alpha ^2\beta , \alpha ^3\beta ^3 )$, whose multiplication is 
\begin{eqnarray*}
a\star b&=&R(a)\cdot \alpha ^3\beta ^2(b)-\alpha ^2\beta ^3(b)\cdot R(\alpha \beta ^{-1}(a))\\
&=&\sum _i\{\alpha \beta ^3(x_i)\cdot (a\cdot \alpha ^3(y_i))\}\cdot \alpha ^3\beta ^2(b)\\
&&\;\;\;\;\;-\sum _i\alpha ^2\beta ^3(b)\cdot \{\alpha \beta ^3(x_i)\cdot (\alpha \beta ^{-1}(a)\cdot \alpha ^3(y_i))\}\\
&\overset{(\ref{BHassoc})}{=}&\sum _i\alpha ^2\beta ^3(x_i)\cdot \{(a\cdot \alpha ^3(y_i))\cdot \alpha ^3\beta (b)\}-
\sum _i\{\alpha \beta ^3(b)\cdot \alpha \beta ^3(x_i)\}\cdot \{\alpha (a)\cdot \alpha ^3\beta (y_i)\}\\
&\overset{(\ref{BHassoc})}{=}&\sum _i\alpha ^2\beta ^3(x_i)\cdot \{\alpha (a)\cdot \alpha ^3(y_i\cdot b)\}-
\sum _i\alpha \beta ^3(b\cdot x_i)\cdot \{\alpha (a)\cdot \alpha ^3\beta (y_i)\}.
\end{eqnarray*}
So, the BiHom-pre-Lie algebras $(A, \ast , \alpha ^2\beta , \alpha ^3\beta ^3)$ and 
$(A, \star , \alpha ^2\beta , \alpha ^3\beta ^3)$ coincide. This extends Aguiar's classical result from \cite{aguiarlectnotes}, 
whose Hom-version was obtained in \cite{lmmp3}, even without the restriction concerning the bijectivity of the 
structure map. 


\begin{center}
ACKNOWLEDGEMENTS
\end{center}

This paper was written while Claudia Menini was a member of the "National Group for Algebraic and 
Geometric Structures and their Applications" (GNSAGA-INdAM). 
Ling Liu was supported by the NSF of China (Grant Nos.11601486, 11801515) and  
Foundation of Zhejiang Educational Committee (Y201738645).

\end{document}